\documentclass[a4paper,12pt]{article}
\usepackage[margin=1in]{geometry}  % set the margins to 1in on all sides

\usepackage{graphicx}              % to include figures
\usepackage{amsmath}
\usepackage{amscd}               % great math stuff
\usepackage{amsfonts} 
\usepackage{amssymb}             % for blackboard bold, etc
\usepackage{amsthm}                % better theorem environment
\usepackage{mathrsfs}
\usepackage{amsbsy}
\usepackage{tikz}
\usepackage{hyperref}

\numberwithin{equation}{section}

\newtheorem{thm}{Theorem}[section]
\newtheorem{defn}[thm]{Definition}
\newtheorem{lem}[thm]{Lemma}
\newtheorem{prop}[thm]{Proposition}
\newtheorem{cor}[thm]{Corollary}

\hypersetup{pdfstartview=}

\makeatletter
\def\th@newremark{\th@remark\thm@headfont{\bfseries}}
\makeatletter

\theoremstyle{newremark}
\newtheorem{rmk}[thm]{Remark}
\newtheorem{eg}[thm]{Example}

\newcommand{\RR}{\mathbb{R}}      % for Real numbers

\newcommand{\cC}{\mathcal{C}}
\newcommand{\dD}{\mathcal{D}}
\newcommand{\eE}{\mathcal{E}}

\newcommand{\lL}{\mathcal{L}}

\newcommand{\nN}{\mathcal{N}}
\newcommand{\oO}{\mathcal{O}}
\newcommand{\pP}{\mathcal{P}}

\newcommand{\sS}{\mathcal{S}}

\newcommand{\wW}{\mathcal{W}}

\newcommand{\s}{\textbf{s}}
\newcommand{\uu}{\textbf{u}}
\newcommand{\vv}{\textbf{v}}

\newcommand*\Bell{\ensuremath{\boldsymbol\ell}}
\newcommand*\Brho{\ensuremath{\boldsymbol\rho}}

\colorlet{symbols}{blue!90!black}
\colorlet{testcolor}{green!60!black}

\usetikzlibrary{shapes.misc}
\usetikzlibrary{shapes.symbols}
\usetikzlibrary{snakes}
\usetikzlibrary{decorations}
\usetikzlibrary{decorations.markings}

\def\drawx{\draw[-,solid] (-3pt,-3pt) -- (3pt,3pt);\draw[-,solid] (-3pt,3pt) -- (3pt,-3pt);}
\tikzset{
	root/.style={circle,fill=testcolor,inner sep=0pt, minimum size=2mm},
	dot/.style={circle,fill=black,inner sep=0pt, minimum size=1mm},
	var/.style={circle,fill=black!10,draw=black,inner sep=0pt, minimum size=
	2mm},
	dotred/.style={circle,fill=black!50,inner sep=0pt, minimum size=2mm},
	generic/.style={semithick,shorten >=1pt,shorten <=1pt},
	dist/.style={ultra thick,draw=testcolor,shorten >=1pt,shorten <=1pt},
	testfcn/.style={ultra thick,testcolor,shorten >=1pt,shorten <=1pt,<-},
	testfcnx/.style={ultra thick,testcolor,shorten >=1pt,shorten <=1pt,<-,
		postaction={decorate,decoration={markings,mark=at position 0.6 with {\drawx}}}},
	kepsilon/.style={semithick,shorten >=1pt,shorten <=1pt,densely dashed,->},
	kprimex/.style={semithick,shorten >=1pt,shorten <=1pt,densely dashed,->,
		postaction={decorate,decoration={markings,mark=at position 0.4 with {\drawx}}}},
	kernel/.style={semithick,shorten >=1pt,shorten <=1pt,->},
	multx/.style={shorten >=1pt,shorten <=1pt,
		postaction={decorate,decoration={markings,mark=at position 0.5 with {\drawx}}}},
	kernelx/.style={semithick,shorten >=1pt,shorten <=1pt,->,
		postaction={decorate,decoration={markings,mark=at position 0.4 with {\drawx}}}},
	kernel1/.style={->,semithick,shorten >=1pt,shorten <=1pt,postaction={decorate,decoration={markings,mark=at position 0.45 with {\draw[-] (0,-0.1) -- (0,0.1);}}}},
	kernel2/.style={->,semithick,shorten >=1pt,shorten <=1pt,postaction={decorate,decoration={markings,mark=at position 0.45 with {\draw[-] (0.05,-0.1) -- (0.05,0.1);\draw[-] (-0.05,-0.1) -- (-0.05,0.1);}}}},
	kernelBig/.style={semithick,shorten >=1pt,shorten <=1pt,decorate, decoration={zigzag,amplitude=1.5pt,segment length = 3pt,pre length=2pt,post length=2pt}},
	gepsilon/.style={dotted,semithick,shorten >=1pt,shorten <=1pt},
	renorm/.style={shape=circle,fill=white,inner sep=1pt},
	labl/.style={shape=rectangle,fill=white,inner sep=1pt},
	xi/.style={circle,fill=symbols!10,draw=symbols,inner sep=0pt,minimum size=1.2mm},
	xix/.style={crosscircle,fill=symbols!10,draw=symbols,inner sep=0pt,minimum size=1.2mm},
	xib/.style={circle,fill=symbols!10,draw=symbols,inner sep=0pt,minimum size=1.6mm},
	xibx/.style={crosscircle,fill=symbols!10,draw=symbols,inner sep=0pt,minimum size=1.6mm},
	not/.style={circle,fill=symbols,draw=symbols,inner sep=0pt,minimum size=0.5mm},
	>=stealth,
	}

\makeatletter
\def\DeclareSymbol#1#2#3{\expandafter\gdef\csname MH@symb@#1\endcsname{\tikz[baseline=#2,scale=0.15,draw=symbols]{#3}}\expandafter\gdef\csname MH@symb@#1s\endcsname{\scalebox{0.7}{\tikz[baseline=#2,scale=0.15,draw=symbols]{#3}}}}
\def\<#1>{\csname MH@symb@#1\endcsname}
\makeatother

\DeclareSymbol{Xi22}{0.5}{\draw (0,0) node[xi] {} -- (-1,1) node[not] {} -- (0,2) node[xi] {}; }

\begin{document}

\title{Inverting the signature of a path}

\author{
Terry J. Lyons\\
\textit{University of Oxford}\\
\and
Weijun Xu\\
\textit{University of Warwick}
}

\maketitle

\abstract{The aim of this article is to develop an explicit procedure that enables one to reconstruct any $\cC^{1}$ path (at natural parametrization) from its signature. We also explicitly quantify the distance between the reconstructed path and the original path in terms of the number of terms in the signature that are used for the construction and the modulus of continuity of the derivative of the path. A key ingredient in the construction is the use of a procedure of symmetrization that separates the behavior of the path at small and large scales. }

\begin{flushleft}

\medskip

\end{flushleft}

\section{Introduction}

\subsection{The signature of a path}

Given an integer $d \geq 1$, a $d$-dimensional path $\gamma$ is a continuous function mapping a closed interval $[0,T]$ into $\RR^{d}$. Throughout the article, we equip $\RR^{d}$ with the $\ell^{1}$ norm. The length of $\gamma$ under this norm is then
\begin{align*}
\|\gamma\| := \sup_{\pP} \sum_{j} |\gamma_{u_{j}} - \gamma_{u_{j-1}}|, 
\end{align*}
where the supremum is taken over all partitions of the interval $[0,T]$, and $|\textbf{x}| := \sum_{i=1}^{d} |x_{i}|$ for $\textbf{x} = (x_{1}, \dots, x_{d}) \in \RR^{d}$. We say $\gamma$ is a path of \textit{finite length} if $\|\gamma\| < +\infty$. 

Given two paths $\alpha: [0,S] \rightarrow \RR^{d}$ and $\beta: [0,T] \rightarrow \RR^{d}$, their \textit{concatenation} $\alpha * \beta$ is a new path defined on the interval $[0,S+T]$ by
\begin{equation} \label{eq:concatenation}
\alpha*\beta(u) := \left \{
\begin{array}{rl}
&\alpha(u), u \in [0,S]\\
&\beta(u-S)+\alpha(S)-\beta(0), u \in [S,S+T]
\end{array} \right.. 
\end{equation}
If $\gamma$ has finite length, then its derivative $\dot{\gamma}$ exists for almost every $t \in [0,T]$. We can re-parametrize $\gamma$ in the fixed time interval $[0,1]$ in such a way that
\begin{align*}
|\dot{\gamma}_{t}| = \sum_{i=1}^{d} |\dot{\gamma}_{t}^{i}| \equiv L, 
\end{align*}
where $\dot{\gamma}^{i}$'s are the components of $\dot{\gamma}$, and $L$ is the length of $\gamma$ under $\ell^{1}$ norm. We call such a parametrization the \textit{natural parametrization} of $\gamma$. 

\begin{rmk}
	The notion of natural parametrization used here is slightly different from the standard one in the literature, as we parametrize $\gamma$ in the unit interval rather than $[0,L]$. As a consequence, the constant speed is $L$ instead of $1$. The reason of using this parametrization is that later on we will compare paths with different lengths, so it will be convenient for us to parametrize all of them in the same interval. 
\end{rmk}

If $\gamma$ has finite length, then one can define a sequence of iterated integrals of it, called the \textit{signature} of $\gamma$. We first introduce the notion of words before giving a precise definition of signature. 

Let $\{e_{1}, \dots, e_{d}\}$ denote the standard basis of $\RR^{d}$. For every integer $n \geq 0$, a word of length $n$ is an ordered sequence of $n$ letters from the set $\{e_{1}, \dots, e_{d}\}$ (with repetition allowed), and we use $|w|$ to denote the length of $w$. For two words $w_{1} = e_{i_{1}} \cdots e_{i_{n}}$ and $w_{2} = e_{j_{1}} \dots e_{j_{m}}$, their \textit{concatenation} $w_{1} * w_{2}$ is a word of length $n+m$ given by
\begin{align*}
w_{1} * w_{2} = e_{i_{1}} \dots e_{i_{n}} e_{j_{1}} \dots e_{j_{m}}. 
\end{align*}
We also use $\emptyset$ to denote the empty word, which is the unique word of length $0$. The signature of a finite length path can now be defined as follows. 

\begin{defn} \label{de:signature}
	Let $\gamma: [0,T] \rightarrow \RR^{d}$ be a path of finite length. For every integer $n$ and every word $w = e_{i_{1}} \cdots e_{i_{n}}$, define
	\begin{align*}
	C_{\gamma}(w) := \int_{0 < u_{1} < \cdots < u_{n} < T} d \gamma_{u_{1}}^{i_{1}} \cdots d \gamma_{u_{n}}^{i_{n}}, 
	\end{align*}
	where $\gamma^{i}$ is the component of $\gamma$ in the direction $e_{i}$. The signature of $\gamma$ is the formal power series
	\begin{align*}
	X(\gamma) = \sum_{n=0}^{+\infty} \sum_{|w|=n} C_{\gamma}(w) w, 
	\end{align*}
	where we have set $C_{\gamma}(\emptyset) = 1$. 
\end{defn}

The signature of $\gamma$ is a definite integral over a fixed time interval where $\gamma$ is defined. Re-parametrizing $\gamma$ does not change its signature. One reason to look at the signatures is that they contain important information about the paths. For example, the first level coefficients $\{C_{\gamma}(w): |w|=1\}$ reproduce the increment of the path, and the second level collection $\{C_{\gamma}(w): |w|=2\}$ represents the area enclosed by the projections of the path on the $e_{i}-e_{j}$ planes. 

The study of the signature dates back to K.T.Chen in 1950's. In a series of papers (\cite{Chen54}, \cite{Chen57}, \cite{Chen58}), he showed that the map $\gamma \mapsto X(\gamma)$ is a homomorphism from the monoid of paths with concatenation to the tensor algebra over $\RR^{d}$, and proved that two piecewise regular paths without backtracks have the same signatures if and only if they differ by a re-parametrization and translation. This uniqueness result, modulo tree-like equivalence, was extended to all finite length paths (in \cite{sig_uniqueness}) and finally to all geometric rough paths\footnote{See \cite{rp_uniqueness} or \cite{rp_bible} for the definition of the signature of a rough path.} (in \cite{rp_uniqueness}). However, all the proofs there are non-constructive, and a natural question that arises after these works is that how one can reconstruct a path from its signature. In this article, we consider the inversion problem when $\gamma$ is $\cC^{1}$ at natural parametrization, which automatically implies that $\gamma$ has finite length and does not have any backtracks in its trajectory.

\begin{rmk} \label{rm:problem}
	In the case $d=1$, all back-forth movements of the path are cancelled out and only the total increment counts. The signature $X$ takes the form $\exp (a e_{1})$, where $a = \gamma_{T} - \gamma_{0}$, and the reduced path with signature $X$ is simply the straight line going from $0$ to $a$. The inversion problem is then interesting only when $d \geq 2$, where evolutions of the path are in general non-commutative. 
\end{rmk}

The key to the reconstruction problem in $d \geq 2$ is then to recover these non-commutative evolutions of the path in their correct orders based on the signature. To appreciate the difficulty of the problem, we note from Definition \ref{de:signature} that each term $C_{\gamma}(w)$ represents some global effect of the path $\gamma$, and a priori there is no obvious way of recovering local information of the path from these global representations. 

A naive approach to the reconstruction problem can be to try to reproduce the path from the interpretation of each term in the signature, as the explanation of the meanings of $\{C_{\gamma}(w): |w| \leq 2\}$ above. However, these intuitive interpretations break down when $|w|$ gets large, and it is essentially impossible to proceed this way to recover much finer information of the path beyond the increment and area. Thus, certain operations on signatures that can reveal local information of the path will be necessary for the reconstruction. 

There have been recent attempts to the reconstruction. In establishing the uniqueness of signature for Brownian motion sample paths, Le Jan and Qian (\cite{Brownsig_uniqueness}) constructed polygonal approximations to the Brownian paths by using the information in their signatures only. This approximation scheme has been further extended to diffusions (\cite{Diffusion_sig}) and a large class of Gaussian processes (\cite{Gaussian_sig}). These works all used a fixed approximation scheme that relies on the law of the underlying random paths which gives certain almost sure non-degeneracy property of the sample paths. However, this non-degeneracy property is not available for all deterministic rough paths, so one would not expect any fixed approximation scheme to work in the deterministic setting without further assumptions on the paths. This obstacle was overcome by Geng in \cite{rp_inversion}, where he allowed flexibility in the scheme used depending on the signature of the path in consideration, and gave a construction of any deterministic geometric rough path from its signature. 

However, all the approximations schemes used in those works are indirect -- they require a sophisticated transform of the signature via smooth one-forms which involves various limiting processes. In addition, even for approximation at any fixed scale and in the setting of smooth paths, these constructions require the use of the \textit{whole signature sequence} rather than any truncation of it. Thus, it is hard to turn them into an effective algorithm. 

In the recent article \cite{inversion_hyperbolic}, by using a construction from hyperbolic geometry, the authors gave an explicit inversion scheme together with stability properties for piecewise linear paths. But this scheme makes an essential use of the piecewise linearity of the path, and it is not clear how it can be extended to more general situations.

\subsection{Main result and strategy}

The main goal of this article is to develop an explicit and implementable procedure that enables one to reconstruct any $\cC^{1}$ path (at natural parametrization) from its signature, together with detailed stability estimates of the procedure. 

As mentioned in the previous subsection (Remark \ref{rm:problem}), the key to establish an effective reconstruction algorithm is to recover the non-commutative evolution of the path in its correct order. For example, the following three paths in $\RR^{2}$ have the same increments. Their signatures agree on level $|w|=1$, but start to differ when $|w| \geq 2$: 
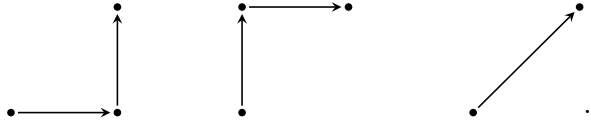
\begin{align*}
\begin{tikzpicture}[scale=0.7,baseline=-0.0cm]
\node at (-2,0) [dot] (left) {}; 
\node at (0,0) [dot] (middle) {}; 
\node at (0,2) [dot] (above) {}; 
\draw[kernel] (left) to (middle); 
\draw[kernel] (middle) to (above); 
\end{tikzpicture}
\qquad \qquad
\begin{tikzpicture}[scale=0.7,baseline=-0.0cm]
\node at (-2,0) [dot] (left) {}; 
\node at (-2,2) [dot] (middle) {}; 
\node at (0,2) [dot] (above) {}; 
\draw[kernel] (left) to (middle); 
\draw[kernel] (middle) to (above); 
\end{tikzpicture}
\qquad \qquad
\begin{tikzpicture}[scale=0.7,baseline=-0.0cm]
\node at (-2,0) [dot] (left) {}; 
\node at (0,2) [dot] (above) {}; 
\draw[kernel] (left) to (above); 
\end{tikzpicture}. 
\end{align*}
Let $x$ and $y$ denote the standard basis of $\RR^{2}$, then the signatures of the above three paths are the formal series $e^{x} e^{y}$, $e^{y} e^{x}$ and $e^{x+y}$, respectively. 

As one can see, the order of the evolution of the paths are captured in the signatures through the ordered letters that consist the words $w$'s (see Definition \ref{de:signature} above). Our main result is that, by symmetrizing the signatures at high level ($|w|$ large), we can average out the non-commutativities of the path at small scales but still preserve the order of its evolution at larger scales. Thus, for any given $k$, we produce from the \textit{truncated signature} a piecewise linear path that approximates the original path at scale $\frac{1}{k}$, and the error of the approximation can be explicitly quantified in terms of the modulus of continuity of the derivative of the path. The main theorem could be loosely stated as follows. 

\begin{thm} \label{th:main_loose}
	Let $\gamma$ be a $\cC^{1}$ path in $\RR^{d}$ at natural parametrization, and assume we know its signature $X = \{C_{\gamma}(w): |w| \geq 0\}$. For large enough $k$ (depending on the dimension and the path, but will be quantified explicitly below), by using the terms in the truncated signature up to level $|w| = 2d k^{3} \log k + k$, we construct a piecewise linear path $\tilde{\gamma}$ with $k$ linear pieces such that when both $\gamma$ and $\tilde{\gamma}$ are at natural parametrization with respect to $\ell^{1}$ norm, for every $t \in [0,1]$ where $\dot{\tilde{\gamma}}_{t}$ is defined, we have
	\begin{equation} \label{eq:main_loose}
	|\dot{\gamma}_{t} - \dot{\tilde{\gamma}}_{t}| < C \eta_{k}, 
	\end{equation}
	where $C > 0$ depends on the path $\gamma$ only, and $\eta_{k} \rightarrow 0$ as $k \rightarrow +\infty$ with rate depending on the modulus of continuity of $\dot{\gamma}$. 
\end{thm}

\begin{rmk}
	A precise statement of this main result, including the value of the constant $C$, can be found in Theorem \ref{th:main} for $d=2$ and in Theorem \ref{th:multi} for higher dimensions. The definition of $\eta_{k}$ is in \eqref{eq:eta} for $d=2$ and in \eqref{eq:eta_multi} for general $d$. As for the decay of $\eta_{k}$, if $\gamma \in \cC^{1,\alpha}$ for some $\alpha \in (0,1]$, then $\eta_{k} = \oO (k^{-\frac{\alpha^{2}}{2}})$ (see \eqref{eq:eta} and Remark \ref{rm:eta} for more details about the rate). But we expect that both this error bound and the level of truncation $|w| \sim k^{3} \log k$ in the signature can be improved. 
\end{rmk}

\begin{rmk}
	Although many arguments below made an explicit use of the continuity of $\dot{\gamma}$, we expect most of them still hold (with a different proportionality constant) if $\dot{\gamma}$ is piecewise continuous, and its discontinuities have jumps less than $\pi$. The only statement that essentially relies on the continuity of $\dot{\gamma}$ is the existence of the word $w^{*}$ with certain non-degeneracy properties (in Theorem \ref{th:concentration}). This word is used in the symmetrization procedure (see Section \ref{sec:symmetrization}) mainly for technical convenience -- one would not have the closed form expression \eqref{eq:expression_symmetrized_sig} for symmetrized signatures without the use of the word to separate blocks in symmetrization. It will be very interesting, and also convenient for practice, if one could prove the validity and stability of the symmetrization procedure without using this word to separate blocks. This would give an inversion scheme for piecewise $\cC^{1}$ paths and possibly all finite length paths. 
	
	On the other hand, however, our procedure does make an essential use of the finite length nature of $\gamma$. Any extension of the procedure to rough path together with detailed quantitative characterization of its stability would require new ingredients. 
\end{rmk}

We now briefly explain the main steps in the construction of such a piecewise linear path. Since there is no essential difference between dimension $2$ and higher, we will mainly focus on the $2$-dimensional case, and give a description of the procedure in higher dimensions with a brief explanation in Section \ref{sec:higher_dimensions}. 

Each line segment in a piecewise linear path is determined by its direction (a unit vector in $\RR^{2}$) and length. Thus, for each integer $k$ and each $j = 1, \dots, k$, our aim is to find a $2$-dimensional unit vector (in $\ell^{1}$ norm)
\begin{align*}
\theta_{j} = \big( a_{j}^{x} \rho_{j}, a_{j}^{y} (1-\rho_{j}) \big), 
\end{align*}
where $\rho_{j} \in [0,1]$ and $a_{j}^{x}, a_{j}^{y} \in \{\pm 1\}$, and $\tilde{L} \in \RR^{+}$ such that
\begin{align*}
\sup_{1 \leq j \leq k} \sup_{u \in [\frac{j-1}{k}, \frac{j}{k}]} |\dot{\gamma}_{u} - \tilde{L} \theta_{j}| < C \eta_{k}. 
\end{align*}
Thus, if we let $\tilde{\gamma}$ to be the path that are concatenations of the line segments $\frac{\tilde{L}}{k} \theta_{j}$'s, this will automatically imply that $\tilde{\gamma}$ is close to the original path $\gamma$ in the Lipschitz norm. 

It is clear that the parameter $\rho_{j} \in [0,1]$ represents the \textit{unsigned direction} of the $j$-th line segment, $a_{j}^{x}, a_{j}^{y}$ represent the signs of $x$ and $y$ directions (in $\RR^{2}$) in that segment, and $\tilde{L}$ is an approximation to the $\ell^{1}$ length of the original path. 

In order to get a rough idea how these parameters can be obtained from the signature of $\gamma$, we first briefly recall from \cite{inversion_hyperbolic} the inversion scheme for integer lattice paths. This can be decomposed into two steps: 

\begin{enumerate}
	\item Identify the unique longest square free word\footnote{A word $w = e_{i_1} \cdots e_{i_n}$ is square free if for all $j = 1, \cdots, n-1$, we have $i_{j} \neq i_{j+1}$. } $w$ such that $C(w) \neq 0$. The order of the letters in $w$ gives the directions (up to the sign) of each piece of the lattice path. 
	\item Move one level up in the signature to recover the sign as well as the length of each piece. 
\end{enumerate}

The reconstruction scheme developed in this article is close in spirit to that for the lattice paths. In what follows, we will recover these parameters in the order of $\rho_{j}$, $a_{j}^{x}$, $a_{j}^{y}$ and then finally $\tilde{L}$. Among these parameters, the most difficult one to recover is the unsigned direction $\rho_{j}$. Ideally, one would like to implement the strategy for integer lattice paths to recover these $\rho_{j}$'s. However, at first glance, that procedure seems to crucially depend on the very special structure of the integer lattice, and does not generalize directly to other situations. In particular, the vanishing/non-vanishing property of coefficients of square free words does not carry over to more general cases where the path can move along any direction in the plane. In order to recover the directions that are not necessarily parallel to the Euclidean axes, we use a symmetrization procedure together with a more robust notion of non-degeneracy that replaces the strict non-zero criterion in Step 1 above. 

Once the directions $\rho_{j}$'s are recovered, one can (similar to the lattice path case) move one level up in the signature to recover the signs $a_{j}^{x}$ and $a_{j}^{y}$. Finally, the approximated length $\tilde{L}$ is obtained by a simple scaling argument. 

In the special case of monotone paths, a much simplified version of the strategy sketched above can be applied to give effective and efficient recovery of the path. This has been carried out in \cite{inversion_monotone}.

\subsection{Assumptions and notations} \label{sec:notations}

We now summarize the assumptions on our path $\gamma$ as well as the notations we will be using in the article. We will mainly work in dimension $d=2$. We equip $\RR^{2}$ with the $\ell^{1}$ norm, and let $x$ and $y$ denote the standard basis elements of $\RR^{2}$. 

The notation $|\cdot|$ will have different meanings in various contexts. If $\alpha \in \RR$, then $|\alpha|$ is the \textit{absolute value} of $\alpha$. If $\alpha = (\alpha_{1}, \alpha_{2}) \in \RR^{2}$ is a vector, then $|\alpha| = |\alpha_{1}| + |\alpha_{2}|$ is its $\ell^{1}$ norm. Finally, if $w$ is a word, then $|w|$ denote the length of $w$. 

Throughout, we fix our path $\gamma$ that is $\cC^{1}$ under \textit{natural parametrization} and has length $L$. More precisely, the path
\begin{align*}
\gamma: [0,1] \rightarrow \RR^{2}, \qquad \gamma_{u} = (x_{u}, y_{u})
\end{align*}
has continuous derivative $\dot{\gamma}_{u} = (\dot{x}_{u}, \dot{y}_{u})$ on $[0,1]$, and satisfies $|\dot{x}_{u}| + |\dot{y}_{u}| \equiv L$ for all $u \in [0,1]$. We let $\delta$ denote the modulus of continuity of $\dot{\gamma}$, so
\begin{equation} \label{eq:modulus}
\delta (\epsilon) := \sup_{|s-t| < \epsilon} |\dot{\gamma}_{s} - \dot{\gamma}_{t}|. 
\end{equation}
For each integer $k$, we define $\epsilon_{k}$ and $\eta_{k}$ to be
\begin{equation} \label{eq:eta}
\epsilon_{k} := \sqrt{2} \bigg( \sqrt{\frac{\delta(1/k)}{L}} + \frac{1}{\sqrt{k}} \bigg), \qquad \eta_{k} := \delta(3 \epsilon_{k}) + \frac{L}{\sqrt{k}}. 
\end{equation}
The definitions \eqref{eq:eta} are for the two dimensional case; the case for general dimension $d$ are introduced in \eqref{eq:eta_multi}. For any integer $k$, let $\Delta_{k-1}$ be the standard simplex
\begin{equation} \label{eq:simplex}
\Delta_{k-1} := \big\{ (u_{1}, \dots, u_{k-1}): 0 = u_{0} < u_{1} < \cdots < u_{k-1} < u_{k} = 1 \big\}. 
\end{equation}
In situations where there might be a confusion, We use boldface letters to denote vectors and normal letters for their components; for example we write
\begin{align*}
\uu = (u_{1}, \dots, u_{k-1}). 
\end{align*}
On the other hand, we write the path as $\gamma = (\gamma^{1}, \gamma^{2})$ as there could be no confusion to arise in this case. 

Given the path $\gamma$ and $\uu \in \Delta_{k-1}$, for every $j = 1, \dots, k$, we let
\begin{equation} \label{eq:increments_direction}
\Delta_{\uu,j} x := x_{u_{j}} - x_{u_{j-1}}, \qquad \Delta_{\uu,j} y := y_{u_{j}} - y_{u_{j-1}}
\end{equation}
denote the increments in the relevant directions in the time interval $[u_{j-1}, u_{j}]$, and
\begin{equation} \label{eq:increments_total}
|\Delta_{\uu,j} \gamma| := |\Delta_{\uu,j} x | + |\Delta_{\uu,j} y|
\end{equation}
be the magnitude of total increments. Similarly, we denote the increments of the $j$-th piece under standard subdivision $\s_{k} = \big\{ \frac{j}{k} \big\}_{j=0}^{k}$ by
\begin{equation} \label{eq:increments_standard}
\Delta_{j} x = x_{j/k} - x_{(j-1)/k}, \qquad \Delta_{j} y = y_{j/k} - y_{(j-1)/k}, 
\end{equation}
and let $|\Delta_{j} \gamma| = |\Delta_{j} x| + |\Delta_{j} y|$. 

\subsection{Organization of the article}

This article is organized as follows. In Section \ref{sec:symmetrization}, we introduce and set up the symmetrization procedure on signatures. Section \ref{sec:concentration} is devoted to the proof of a concentration property of the symmetrized signatures with explicit quantitative estimates. These estimates will then be used in Section \ref{sec:reconstruction} where we give a detailed description of the reconstruction procedure of the path from the signatures. Finally, in Section \ref{sec:higher_dimensions}, we give an outline of the reconstruction procedure for paths in higher dimensions.

\begin{flushleft}
\textbf{Acknowledgements}
\end{flushleft}
We thank two referees for carefully reading the manuscript and providing helpful suggestions. Weijun Xu also thanks Horatio Boedihardjo for many helpful discussions. 

The research of Terry Lyons is supported by EPSRC grant EP/H000100/1 and the European Research Council under the European Union’s Seventh Framework Program (FP7-IDEAS-ERC) / ERC grant agreement nr. 291244. Terry Lyons acknowledges the support of the Oxford-Man Institute. Weijun Xu has been supported by the Oxford-Man Institute through a scholarship during his time as a student at Oxford. He is now supported by Leverhulme trust.

\section{Symmetrization} \label{sec:symmetrization}

As mentioned in the introduction, the first step in the reconstruction is to obtain the unsigned direction $\rho_{j} \in [0,1]$ of the $j$-th line segment in our piecewise linear path approximation, and the key to the accurate recovery of this parameter is the use of symmetrization procedure together with a robust notion of non-degeneracy. We first give a simple example to illustrate this.

\begin{eg} \label{eg:increment}
Let $X$ be the signature of some bounded variation path $\gamma_{t} = (x_{t},y_{t}), t \in [0,1]$, and we would like to recover from $X$ the increments
\begin{align*}
\Delta x:= x_{1} - x_{0} \qquad \text{and} \qquad \Delta y := y_{1} - y_{0}, 
\end{align*}
where $(x_{0},y_{0})$ and $(x_{1},y_{1})$ are beginning and end points of $\gamma$. One could of course get the exact values of the pair $(\Delta x, \Delta y)$ directly from the first level signature $X^{1}$, but cannot proceed much further. The symmetrization method given below is more complicated, but has the advantage that it can be generalized to recover much more information beyond the increments. 

The symmetrization procedure to recover the increments is as follows. For every integer $n$ and every $\ell = 0, 1, \dots, n$, let
\begin{align*}
\sS^{n}(\ell) := n! \sum C(w), 
\end{align*}
where the sum is taken over all words $w$ with length $n$ that contain $\ell$ $x$'s and $(n-\ell)$ $y$'s. Then, $\sS^{n}(\ell)$ has the expression
\begin{align*}
\sS^{n}(\ell) = \begin{pmatrix} n \\ \ell \end{pmatrix} (\Delta x)^{\ell} (\Delta y)^{n-\ell}. 
\end{align*}
Note that for each $n$ and $\ell$, the left hand side above is the information available to us (from the signature), and the right hand side is its expression. It is standard that for fixed large $n$, the quantity $|\sS^{n}(\ell)|$ is maximized near the value $\ell^{*}$ such that
\begin{align*}
\frac{\ell^{*}}{n - \ell^{*}} \approx \frac{|\Delta x|}{|\Delta y|}. 
\end{align*}
We can then asymptotically recover the ratio $|\Delta x| : |\Delta y|$ by finding the maximizer $\ell^{*}$ of $\sS^{n}(\ell)$, and this gives us the \textit{unsigned direction} of the increment. 

To recover the signs of $\Delta x$ and $\Delta y$, one repeats the same trick as in the case for integer lattice paths: moving one level up and comparing the signs. Finally, the magnitude of the increment will be obtained via scaling. 
\end{eg}

The example above illustrates the case for the single-piece approximation to the path (that is, $k=1$). In order to recover finer information of the path (for large $k$), instead of symmetrizing the whole signature, we divide high level signatures into $k$ equal blocks, and symmetrize each block. We first introduce some notations before we set up the procedure. 

For every integer $k$ and $n$, we let $\lL^{n}_{k}$ denote the set of multi-indices
\begin{equation} \label{eq:index_set}
\lL^{n}_{k} := \bigg\{ \Bell = (\ell_{1}, \dots, \ell_{k}): 0 \leq \ell_{j} \leq n \bigg\}. 
\end{equation}
For every word $w$, we let $|w|_{x}$ and $|w|_{y}$ denote the number of letters $x$ and $y$ in $w$, respectively. For every $w$ of length $k-1$ of the form $w = e_{i_{1}} \cdots e_{i_{k-1}}$ and every multi-index $\Bell \in \lL^{n}_{k}$, we let $\wW^{2n}_{k}(w, \Bell)$ be the set of words
\begin{equation} \label{eq:blocks_word}
\wW^{2n}_{k}(w, \Bell) = \bigg\{w' = w_{1} * e_{i_{1}} * \cdots * e_{i_{k-1}} * w_{k}: \phantom{1} |w_{j}|_{x} = 2 \ell_{j}, |w_{j}|_{y} = 2n - 2 \ell_{j} \bigg\}. 
\end{equation}
A typical word $w' \in \wW^{2n}_{k}(w, \Bell)$ where $w = e_{i_{1}} \dots e_{i_{k-1}}$ has the form 
\begin{align*}
\underbrace{*****}_{w_{1}} \phantom{1} e_{i_{1}} \phantom{1} \underbrace{*****}_{w_{2}} \phantom{1} e_{i_{2}} \dots \dots e_{i_{k-2}} \underbrace{*****}_{w_{k-1}} \phantom{1} e_{i_{k-1}} \phantom{1} \underbrace{*****}_{w_{k}}. 
\end{align*}
Here, each $w_{j}$ is a sub-word of length $2n$ with $2 \ell_{j}$ letters $x$ and $2n - 2\ell_{j}$ letters $y$. The two consecutive sub-words (blocks) $w_{j}$ and $w_{j+1}$ are separated by the letter $e_{i_{j}}$ from $w$. For example, for $n=2$ and $k=1$, $\lL^{2}_{1} = \{0,1,2\}$, so we have
\begin{align*}
&\wW^{4}_{1}(\emptyset, (0)) = \big\{ yyyy \big\}, \qquad \wW^{4}_{1}(\emptyset, (2)) = \big\{ xxxx \big\}, \\
&\wW^{4}_{1}(\emptyset, (1)) = \big\{ xxyy, xyxy, xyyx, yxxy, yxyx, yyxx \big\}. 
\end{align*}
For $n=k=2$, the set $\wW^{4}_{2}(x, (1,0))$ consists of the words
\begin{align*}
\big\{ xxyyxyyyy, xyxyxyyyy, xyyxxyyyy, yxxyxyyyy, yxyxxyyyy, yyxxxyyyy \big\}. 
\end{align*}
The set $\wW^{4}_{2}(x, (1,2))$ is similar except that one replaces the last four $y$'s by four $x$'s. With this definition, we introduce the \textit{symmetrized signatures}
\begin{equation} \label{eq:symmetrized_sig}
\sS^{2n}_{k}(w, \Bell) := ((2n)!)^{k} \sum_{w' \in \wW^{2n}_{k}(w, \Bell)} C(w'). 
\end{equation}
Recall the definitions of $\Delta_{k-1}$, $\Delta_{\uu,j} x$ and $\Delta_{\uu,j} y$ in \eqref{eq:simplex} and \eqref{eq:increments_direction}, we have the following proposition. 

\begin{prop} \label{pr:expression_symmetrized_sig}
Fix integer $k$ and $n$. Let $w = e_{i_{1}} \cdots e_{i_{k-1}}$ and $\Bell = \{ \ell_{1}, \dots, \ell_{k} \} \in \lL^{n}_{k}$. Then, the quantity $\sS^{2n}_{k}(w, \Bell)$ defined above has the expression
\begin{equation} \label{eq:expression_symmetrized_sig}
\sS^{2n}_{k}(w, \Bell) = \int_{\Delta_{k-1}} \prod_{j=1}^{k-1} \dot{\gamma}_{u_{j}}^{i_{j}} \prod_{j=1}^{k} \begin{pmatrix} 2n \\ 2\ell_{j} \end{pmatrix} (\Delta_{\uu,j} x)^{2 \ell_{j}} (\Delta_{\uu,j} y)^{2n - 2\ell_{j}} d\uu. 
\end{equation}
\end{prop}
\begin{proof}
	By the definition of $\wW^{2n}_{k}(w, \Bell)$ in \eqref{eq:blocks_word}, we have
	\begin{align*}
	&\phantom{111}\sum_{w' \in \wW^{2n}_{k}(w, \Bell)} C(w')\\
	&= \int_{\uu \in \Delta_{k-1}} \prod_{j=1}^{k} \bigg( \frac{1}{(2 \ell_{j})! (2n-2\ell_{j})!} \int_{u_{j-1} < v_{1}^{j}, \dots, v_{2n}^{j} < u_{j}} d \gamma_{v_{1}^{j}}^{i_{1}^{j}} \cdots d\gamma_{v_{2n}^{j}}^{i_{2n}^{j}} \bigg) d \gamma^{i_{1}}_{u_{1}} \cdots d \gamma^{i_{k-1}}_{u_{k-1}}, 
	\end{align*}
	where $\uu = (u_{1}, \dots, u_{k-1})$, $u_{0} = 0$, $u_{k} = 1$, and for each $j$, the word $w_{j} = e_{i_{1}^{j}} \cdots e_{i_{2n}^{j}}$ consists of $2 \ell_{j}$ letters $x$ and $2n-2\ell_{j}$ letters $y$. Thus, we have
	\begin{align*}
	\int_{u_{j-1} < v_{1}^{j}, \dots, v_{2n}^{j} < u_{j}} d \gamma_{v_{1}^{j}}^{i_{1}^{j}} \cdots d\gamma_{v_{2n}^{j}}^{i_{2n}^{j}} = \big( \Delta_{\uu,j} x \big)^{2 \ell_{j}} \big( \Delta_{\uu,j} y \big)^{2n - 2 \ell_{j}}, 
	\end{align*}
	and \eqref{eq:expression_symmetrized_sig} follows immediately. 
\end{proof}

We will see below that these $\sS^{2n}_{k}(w, \Bell)$'s are the only quantities we will use to recover the unsigned directions of each piece in our piecewise linear approximation. It is the recovery of the sign of each direction that requires more information in the signature other than the $\sS^{2n}_{k}$'s. We will introduce those additional quantities only when it becomes necessary.

\begin{rmk}
We emphasize that \eqref{eq:symmetrized_sig} is the definition of the symmetrized signature; this information is available to us from the signature $X$. On the other hand, \eqref{eq:expression_symmetrized_sig} is an expression of this quantity, and we will make use of this expression later to prove a priori bounds of the symmetrized signature. 
\end{rmk}

\begin{rmk}
The reason why we insert a letter $e_{i_{j}}$ between every two consecutive symmetrized blocks is to let $\sS^{2n}_{k}(w, \Bell)$ have a closed form expression as in \eqref{eq:expression_symmetrized_sig}. This is mainly for technical convenience, and we expect results in the next section still hold true when the symmetrization is taken without using these $e_{i_{j}}$'s to separate blocks. Also, the symmetrization is taken only over even numbers of $x$'s and $y$'s in each block. This is to avoid cancellations of different signs inside the integration on the right hand side of \eqref{eq:expression_symmetrized_sig}. 
\end{rmk}

Before we proceed, we first give a heuristic explanation how the quantities $\sS^{2n}_{k}(w, \Bell)$'s can help recover the unsigned directions of each segment of the path. For fixed word $w = e_{i_{1}} \cdots e_{i_{k-1}}$, using Proposition \ref{pr:expression_symmetrized_sig} and then summing over $\Bell \in \lL^{n}_{k}$, we have
\begin{equation} \label{eq:sum_symmetrization}
\sum_{\Bell \in \lL^{n}_{k}} \sS^{2n}_{k}(w, \Bell) = \frac{1}{2^{k}} \int_{\Delta_{k-1}} \prod_{j=1}^{k-1} \dot{\gamma}_{u_{j}}^{i_{j}} \prod_{j=1}^{k} \big( (\Delta_{\uu,j} x + \Delta_{\uu,j} y)^{2n} + (\Delta_{\uu,j} x - \Delta_{\uu,j} y)^{2n} \big) d \uu. 
\end{equation}
If $n \gg k$, it is natural to expect that the above integrand has magnitude of order $C_{k} \prod_{j=1}^{k} |\Delta_{\uu,j} \gamma|^{2n}$, where $|\Delta_{\uu,j} \gamma| = |\Delta_{\uu,j} x| + |\Delta_{\uu,j} y|$ is defined in \eqref{eq:increments_total}. This suggests
\begin{equation} \label{eq:maximal_integral}
\sum_{w: |w|=k-1} \sum_{\Bell \in \lL^{n}_{k}} \sS^{2n}_{k}(w, \Bell) \sim C_{k} \int_{\Delta_{k-1}} \prod_{j=1}^{k} |\Delta_{\uu,j}  \gamma|^{2n} d \uu. 
\end{equation}
By the concentration properties of the integral $\int_{\Delta_{k-1}} \prod_{j} |\Delta_{\uu,j} \gamma|^{2n} d \uu$ (which will be proven and quantified in the next Section) as well as the normal approximation to binomials, it turns out that as long as we sum the $\sS^{2n}_{k}(w, \Bell)$'s over a very small range of $\Bell \in \lL^{n}_{k}$ such that
\begin{equation} \label{eq:range_sum_direction}
\frac{\ell_{j}}{n-\ell_{j}} \sim \frac{|\Delta_{j} x|}{|\Delta_{j} y|}, \qquad j = 1, \dots, k, 
\end{equation}
then the sum will also be ``close" to its possible maximum -- the right hand side of \eqref{eq:maximal_integral}. On the other hand, if any $\ell_{j}$ is away from the range in \eqref{eq:range_sum_direction}, then the sum of $\sS^{2n}_{k}(w, \Bell)$'s will be negligible compared to \eqref{eq:maximal_integral}. Then, similar as Example \ref{eg:increment}, ``observing" which range of $\ell_{j}$ maximizes the sum of $\sS^{2n}_{k}(w, \Bell)$'s will recover asymptotically the unsigned directions of each piece. We will precisely formulate and prove this heuristic in the next sections.

\section{Concentration of symmetrized signatures} \label{sec:concentration}

The aim of this section is to prove a quantitative statement about the concentration property of the integral
\begin{equation} \label{eq:simple_integral}
\int_{\Delta_{k-1}} \prod_{j=1}^{k} |\Delta_{\uu,j} \gamma|^{n} d \uu
\end{equation}
when $n \gg k$, where $|\Delta_{\uu,j} \gamma|$ is the magnitude of the increment of $\gamma$ in $[u_{j-1}, u_{j}]$ as defined in \eqref{eq:increments_total}. This concentration property roughly states that although the integration is taken over the whole simplex, when $n$ is large, almost all its contribution comes from a very small subset of $\Delta_{k-1}$. In fact, the domain of concentration is around the points $\uu \in \Delta_{k-1}$ such that the product $\prod_{j} |\Delta_{\uu,j} \gamma|$ is maximized, and these maximizers cannot be far away from the standard dissection $\s_{k} = \{\frac{j}{k}\}_{j=1}^{k-1}$. A quantitative statement will be given in Proposition \ref{pr:concentration} below. As a consequence, we will obtain in Theorem \ref{th:concentration} a precise quantitative characterization of the ``heuristic" \eqref{eq:maximal_integral}. Throughout, we assume $\gamma: [0,1] \rightarrow \RR^{2}$ is at natural parametrization with respect to $\ell^{1}$ norm, $\gamma \in \cC^{1}$, and $\delta$ is the modulus of continuity of $\dot{\gamma}$.

\begin{lem} \label{le:standard_location}
Let $k$ be an integer such that $\delta(\frac{1}{k}) < \frac{L}{2}$, where $L$ is the $\ell^{1}$ length of $\gamma$. Then, for every $j = 1, \dots, k$, we have
\begin{align*}
\frac{L - \delta(\frac{1}{k})}{k} \leq |\Delta_{j} \gamma| \leq \frac{L}{k}, 
\end{align*}
where $|\Delta_{j} \gamma| = |\Delta_{j} x| + |\Delta_{j} y|$ is the magnitude of the increment of $\gamma$ in $[\frac{j-1}{k}, \frac{j}{k}]$, as defined in \eqref{eq:increments_standard}. 
\end{lem}
\begin{proof}
The inequality $|\Delta_{j} \gamma| \leq \frac{L}{k}$ follows immediately from the assumption that $\gamma$ is at natural parametrization. 

For the lower bound, for each $j = 1, \dots, k$, we let $I_{j} = [\frac{j-1}{k}, \frac{j}{k}]$. If both $\dot{x}_{u}$ and $\dot{y}_{u}$ keep their signs unchanged in the interval $I_{j}$, then we have
\begin{align*}
\bigg| \int_{I_{j}} \dot{x}_{u} du \bigg| = \int_{I_{j}} |\dot{x}_{u}| du, \qquad \bigg| \int_{I_{j}} \dot{y}_{u} du \bigg| = \int_{I_{j}} |\dot{y}_{u}| du, 
\end{align*}
and it follows immediately that $|\Delta_{j} \gamma| = \frac{L}{k}$. If not, then either $\dot{x}$ or $\dot{y}$ is $0$ at some point $u$ in the interval $I_{j}$. Suppose without loss of generality that $\dot{y}_{u} = 0$ for some $u \in I_{j}$, then we have
\begin{align*}
\sup_{u \in I_{j}} |\dot{y}_{u}| \leq \delta \big(\frac{1}{k}\big), 
\end{align*}
which in turn gives
\begin{align*}
|\dot{x}_{u}| \geq L - \delta \big(\frac{1}{k}\big) \geq \frac{L}{2}. 
\end{align*}
for all $u \in I_{j}$. In particular, the continuity of $\dot{x}_{u}$ implies that it does not change sign in $I_{j}$. Thus, we have
\begin{align*}
|\Delta_{j} \gamma| \geq \bigg| \int_{I_{j}} \dot{x}_{u} du \bigg| = \int_{I_{j}} |\dot{x}_{u}| du \geq \frac{L - \delta(\frac{1}{k})}{k}. 
\end{align*}
This finishes the proof of the lemma. 
\end{proof}

Lemma \ref{le:standard_location} says that at $\uu  = \s_{k} = \{ \frac{j}{k} \}$, the product $\prod_{j} |\Delta_{j} \gamma|$ is close to its largest possible value. On the other hand, if any $u_{j}$ is far away from $\frac{j}{k}$, then $\prod_{j} |\Delta_{\uu,j} \gamma|$ must be small. This is the content of the following lemma.

\begin{lem} \label{le:no_deviation}
Let $k$ be an integer such that $\delta(\frac{1}{k}) < \frac{L}{2}$, and let $\uu = \{ u_{1}  < \cdots < u_{k-1} \} \in \Delta_{k-1}$. Recall from \eqref{eq:eta} that
\begin{align*}
\epsilon_{k} = \sqrt{2} \bigg( \sqrt{\frac{\delta(\frac{1}{k})}{L}} + \frac{1}{\sqrt{k}} \bigg). 
\end{align*}
If $\big|u_{j} - \frac{j}{k} \big| \geq \epsilon_{k}$ for some $1 \leq j \leq k$, then
\begin{equation} \label{eq:no_deviation}
\prod_{i=1}^{k} \bigg( |\Delta_{\uu,i} \gamma| \big/ |\Delta_{i} \gamma| \bigg) < \frac{1}{e}. 
\end{equation}
\end{lem}
\begin{proof}
Since the left hand side of \eqref{eq:no_deviation} is invariant under rescaling of length, we can assume without loss of generality that $L = 1$. Suppose
\begin{align*}
u_{j} - \frac{j}{k} = \epsilon
\end{align*}
for some $j$ and some $\epsilon$. Then $u_{j} = \frac{j}{k} + \epsilon$, and the sum of all increments before and after the time $t = u_{j}$ satisfy
\begin{align} \label{eq:piece_bound}
\sum_{i=1}^{j} |\Delta_{\uu,i} \gamma| \leq \frac{j}{k} + \epsilon, \qquad \sum_{i=j+1}^{k} |\Delta_{\uu,i} \gamma| \leq \frac{k-j}{k} - \epsilon. 
\end{align}
Note that here we do not require $\epsilon$ to be positive. By the bound \eqref{eq:piece_bound}, the best possible maximum one can hope for $\prod_{i} |\Delta_{\uu,i} \gamma|$ is the case when we have
\begin{align*}
|\Delta_{\uu,i} \gamma| = \frac{1}{k} + \frac{\epsilon}{j}, \phantom{11} \forall i \leq j \qquad \text{and} \qquad |\Delta_{\uu,i} \gamma| = \frac{1}{k} - \frac{\epsilon}{k-j}, \phantom{11} \forall i \geq j+1, 
\end{align*}
which gives
\begin{align*}
\prod_{i=1}^{k} |\Delta_{\uu,i} \gamma| \leq \bigg( \frac{1}{k} + \frac{\epsilon}{j} \bigg)^{j} \bigg( \frac{1}{k} - \frac{\epsilon}{k-j} \bigg)^{k-j}. 
\end{align*}
Since $\delta(\frac{1}{k}) < \frac{L}{2}$, we can apply Lemma \ref{le:standard_location} to get
\begin{equation} \label{eq:product_ratio}
\prod_{i=1}^{k} \bigg(|\Delta_{\uu,i} \gamma| \big/ |\Delta_{i} \gamma| \bigg) \leq \bigg( \frac{(1 + p \epsilon)^{\frac{1}{p}} (1 - q \epsilon)^{\frac{1}{q}}}{1 - \delta(\frac{1}{k})} \bigg)^{k}, 
\end{equation}
where $p = \frac{k}{j}$ and $q = \frac{k}{k-j}$. Now, let
\begin{align*}
f(x) = (1 + px)^{\frac{1}{p}} (1 - qx)^{\frac{1}{q}}, 
\end{align*}
then $f$ is defined on the interval $[-\frac{1}{p}, \frac{1}{q}] = [-\frac{j}{k}, 1 - \frac{j}{k}]$, and has derivative
\begin{align*}
f'(x) = - (p+q) (1+px)^{-\frac{1}{q}} (1-qx)^{-\frac{1}{p}} x. 
\end{align*}
For $x \in [0,\frac{1}{q}]$, $f'(x)$ is negative and satisfies
\begin{align*}
|f'(x)| \geq (p+q) \big(1 + \frac{p}{q} \big)^{-\frac{1}{q}} x = (p+q)^{\frac{1}{p}} q^{\frac{1}{q}} x \geq x. 
\end{align*}
Similarly, for $x \in [-\frac{1}{p}, 0]$, $f'(x)$ is positive and satisfies $f'(x) \geq |x|$. Noting that $f(0) = 1$, we then deduce
\begin{align*}
f(x) \leq 1 - \frac{x^{2}}{2}, \qquad x \in [-\frac{1}{p}, \frac{1}{q}]. 
\end{align*}
Plugging this bound into \eqref{eq:product_ratio} with $x=\epsilon$ (noting that $\epsilon$ could be negative), we get
\begin{align*}
\prod_{i=1}^{k} \bigg(|\Delta_{\uu,i} \gamma| \big/ |\Delta_{i} \gamma| \bigg) \leq \bigg( \frac{1 - \frac{\epsilon^{2}}{2}}{1 - \delta(\frac{1}{k})} \bigg)^{k}. 
\end{align*}
It is then clear that if $\frac{\epsilon^{2}}{2} > \delta(\frac{1}{k}) + \frac{1}{k}$, we will necessarily have
\begin{align*}
\prod_{i=1}^{k} \bigg(|\Delta_{\uu,i} \gamma| \big/ |\Delta_{i} \gamma| \bigg) \leq \bigg( 1 - \frac{1}{k} \bigg)^{k} < \frac{1}{e}. 
\end{align*}
The case for general length $L$ is essentially the same except one replaces $\delta(\frac{1}{k})$ by $\delta(\frac{1}{k}) / L$. 
\end{proof}

In view of Lemma \ref{le:no_deviation}, we let $E_{k-1}$ be the set
\begin{equation} \label{eq:concentration_set}
E_{k-1} = \bigg\{ (u_{1}, \cdots, u_{k-1}): \big|u_{j} - \frac{j}{k}\big| < \epsilon_{k}, j = 1, \cdots ,k-1 \bigg\}. 
\end{equation}
We can now prove the following concentration property.

\begin{prop} \label{pr:concentration}
For every integer $k$ such that $\delta(\frac{1}{k}) < \frac{L}{2}$ and every integer $n$, we have
\begin{equation} \label{eq:concentration}
\int_{\Delta_{k-1} \cap E_{k-1}} \prod_{j=1}^{k} |\Delta_{\uu,j} \gamma|^{n} d\uu \geq \big(1 - e^{3 k \log k - \frac{n}{2}} \big) \int_{\Delta_{k-1}} \prod_{j=1}^{k} |\Delta_{\uu,j} \gamma|^{n} d\uu. 
\end{equation}
\end{prop}
\begin{proof}
Let $k \geq K$, and let $\eE_{k-1}$ denote the set
\begin{align*}
\eE_{k-1} = \bigg\{ \vv = (v_{1}, \dots, v_{k-1}): \big| v_{j} - \frac{j}{k} \big| < \frac{1}{12 k^{2}}, \phantom{1} j = 1, \dots, k-1 \bigg\}. 
\end{align*}
Then, for any $\vv \in \eE_{k-1}$ and $j = 1, \dots, k$, we have
\begin{equation} \label{eq:increment_difference}
\big| |\Delta_{\vv,j} \gamma| - |\Delta_{j} \gamma| \big| \leq \big| \gamma_{v_{j}} - \gamma_{\frac{j}{k}} \big| + \big| \gamma_{v_{j-1}} - \gamma_{\frac{j-1}{k}} \big| < \frac{L}{6 k^{2}}. 
\end{equation}
Since $\delta(\frac{1}{k}) < \frac{L}{2}$, we have $L \leq 2 (L - \delta(\frac{1}{k}))$, so it follows from \eqref{eq:increment_difference} and then Lemma \ref{le:standard_location} that
\begin{align*}
|\Delta_{\vv,j} \gamma| \geq |\Delta_{j} \gamma| - \frac{L - \delta(\frac{1}{k})}{3 k^{2}} \geq \big( 1 - \frac{1}{3k} \big) |\Delta_{j} \gamma|. 
\end{align*}
Since $(1 - \frac{1}{3k})^{k}$ is increasing in $k$, multiplying over $j$ from $1$ to $k$, we get
\begin{equation} \label{eq:comparison_domains}
\prod_{j=1}^{k} |\Delta_{\vv,j} \gamma| \geq \frac{2}{3} \prod_{j=1}^{k} |\Delta_{j} \gamma| > e^{-\frac{1}{2}} \prod_{j=1}^{k} |\Delta_{j} \gamma|
\end{equation}
for all $\vv \in \eE_{k-1}$. Now, raising both sides of \eqref{eq:comparison_domains} to power $n$ and using Lemma \ref{le:no_deviation}, we get
\begin{align*}
\prod_{j=1}^{k} |\Delta_{\uu,j} \gamma|^{n} < e^{-\frac{n}{2}} \prod_{j=1}^{k} |\Delta_{\vv,j} \gamma|^{n}, \qquad \forall \phantom{1} \uu \in \Delta_{k-1} \cap E_{k-1}^{c}, \vv \in \eE_{k-1}. 
\end{align*}
Averaging both sides above in their respective domains, and noting $\eE_{k-1} \subset \Delta_{k-1}$, we deduce that
\begin{equation} \label{eq:bound_negligible}
\int_{\Delta_{k-1} \cap E_{k-1}^{c}} \prod_{j=1}^{k} |\Delta_{\uu,j} \gamma|^{n} d \uu < C_{k} e^{-\frac{n}{2}} \int_{\Delta_{k-1}} \prod_{j=1}^{k} |\Delta_{\uu,j} \gamma|^{n} d \uu, 
\end{equation}
where the constant $C_{k}$ is given by
\begin{align*}
C_{k} = \frac{|\Delta_{k-1} \cap E_{k-1}^{c}|}{|\eE_{k-1}|}, 
\end{align*}
and $|\cdot|$ denotes the volume of a set. Since $|\Delta_{k-1} \cap E_{k-1}^{c}| < |\Delta_{k-1}| = \frac{1}{(k-1)!}$ and $|\eE_{k-1}| = (6k^{2})^{-(k-1)}$, using the bound
\begin{align*}
\log (k!) > \int_{1}^{k} \log x dx = k \log k - (k-1)
\end{align*}
and the fact that $\big(1 + \frac{1}{k-1} \big)^{k-1} < e$ for all $k$, we obtain the estimate
\begin{align*}
C_{k} < (6ek)^{k-1}. 
\end{align*}
The conclusion of the theorem follows by plugging this estimate of $C_{k}$ into \eqref{eq:bound_negligible} and noting that $(k-1) \log (6ek) < 3 k \log k$. 
\end{proof}

Proposition \ref{pr:concentration} (and also Theorem \ref{th:concentration} below) will be useful when $n \gg k \log k$, in which case the integral $\int_{\Delta_{k-1}} \prod_{j=1}^{k} |\Delta_{\uu,j} \gamma|^{k}$ is concentrated on the sub-domain $\Delta_{k-1} \cap E_{k-1}$. As a consequence, the sum of the symmetrized signatures $\sS^{2n}(w, \Bell)$'s cannot be too far away from its maximal possible value as $n \rightarrow +\infty$. This is the content of the following theorem.

\begin{thm} \label{th:concentration}
For every $k$ such that $\delta(2 \epsilon_{k})  < \frac{L}{6}$, there exists a word $w^{*}$ of length $k-1$ such that for every $n$, we have
\begin{equation} \label{eq:non_degenerate}
\sum_{\Bell \in \lL^{n}_{k}} \big| \sS^{2n}_{k}(w^{*}, \Bell) \big| \geq \bigg( \frac{3}{6^{k}} - \big(1 + \frac{3}{6^{k}}\big) e^{3k \log k-n} \bigg) L^{k-1} \int_{\Delta_{k-1}} \prod_{j=1}^{k} |\Delta_{\uu,j} \gamma|^{2n} d\uu, 
\end{equation}
where $\lL^{n}_{k}$ is set of multi-indices as defined in \eqref{eq:index_set}. 
\end{thm}
\begin{proof}
For any word $w$ of length $k-1$, by \eqref{eq:expression_symmetrized_sig}, we have
\begin{equation} \label{eq:summing_indices}
\sum_{\Bell \in \lL^{n}_{k}} \big| \sS^{2n}_{k}(w, \Bell) \big| \geq \bigg| \int_{\Delta_{k-1}} \prod_{j=1}^{k-1} \dot{\gamma}_{u_{j}}^{i_{j}}  \prod_{j=1}^{k} \sum_{\ell_{j}=0}^{n} \begin{pmatrix} 2n \\ 2 \ell_{j}  \end{pmatrix} (\Delta_{\uu,j}x)^{2 \ell_{j}} (\Delta_{\uu,j} y)^{2n - 2 \ell_{j}}  d\uu  \bigg|, 
\end{equation}
where we have interchanged the sum over $\Bell$ and the product over $j$ since different components of $\Bell$ are summed up independently. The integrand of the right hand side of \eqref{eq:summing_indices} can be split into two parts: the product of pointwise derivatives $\dot{\gamma}_{u_{j}}^{i_{j}}$, whose ``direction" $i_{j}$ is given by the $j$-th letter in the word $w$, and the product of the increments $\sum_{\ell_{j}} \begin{pmatrix} 2n \\ 2\ell_{j}  \end{pmatrix} (\Delta_{\uu,j} x)^{2 \ell_{j}} (\Delta_{\uu,j} y)^{2n - 2\ell_{j}}$. For the latter one, since
\begin{align*}
\sum_{\ell_{j}=0}^{n} \begin{pmatrix} 2n \\ 2\ell_{j}  \end{pmatrix} (\Delta_{\uu,j} x)^{2 \ell_{j}} (\Delta_{\uu,j} y)^{2n - 2\ell_{j}} = \frac{1}{2} \bigg( (\Delta_{\uu,j}x + \Delta_{\uu,j}y)^{2n} + (\Delta_{\uu,j}x - \Delta_{\uu,j}y)^{2n} \bigg), 
\end{align*}
which is bounded from above by $|\Delta_{\uu,j} \gamma|^{2n}$ and from below by $\frac{1}{2} |\Delta_{\uu,j} \gamma|^{2n}$, we have the bound
\begin{equation} \label{eq:second_part}
\frac{1}{2^{k}} \prod_{j=1}^{k} |\Delta_{\uu,j} \gamma|^{2n} \leq  \prod_{j=1}^{k} \sum_{\ell_{j}=0}^{n} \begin{pmatrix} 2n \\ 2\ell_{j}  \end{pmatrix} (\Delta_{\uu,j} x)^{2 \ell_{j}} (\Delta_{\uu,j} y)^{2n - 2\ell_{j}} \leq \prod_{j=1}^{k} |\Delta_{\uu,j} \gamma|^{2n}. 
\end{equation}
Now we look at the first part of the integrand, $\prod_{j=1}^{k-1} \dot{\gamma}_{u_{j}}^{i_{j}}$. Since we hope the whole integral on the right hand side of \eqref{eq:summing_indices} to be concentrated on the domain $\Delta_{k-1} \cap E_{k-1}$, we choose a word $w^{*} = e_{i_{1}} \cdots e_{i_{k-1}}$ such that
\begin{equation} \label{eq:first_part}
|\dot{\gamma}_{u_{j}}^{i_{j}}| \geq \frac{L}{3}
\end{equation}
for all $j$ and all $\uu \in E_{k-1}$. This also guarantees that none of the $\dot{\gamma}_{u_{j}}^{i_{j}}$'s changes its sign in this domain. The main purpose of choosing $w^{*}$ in this way is to ensure that the term $\prod_{j} \dot{\gamma}_{u_{j}}^{i_{j}}$ does not cause any degeneracy or cancellations of the integral in its domain of concentration $\Delta_{k-1} \cap E_{k-1}$. The continuity of $\dot{\gamma}$ ensures that we can always find such a word as long as $\delta(2 \epsilon_{k}) < \frac{L}{6}$. 

We now decompose the the integral on the right hand side of \eqref{eq:summing_indices} into two disjoint domains: $\Delta_{k-1} \cap E_{k-1}$ and $\Delta_{k-1} \cap E_{k-1}^{c}$. For the first one, since the product $\prod_{j=1}^{k-1} \dot{\gamma}_{u_{j}}^{i_{j}}$ is bounded away from $0$ by $(L/3)^{k-1}$ and does not change its sign in $E_{k-1}$, and the rest of the integrand is always positive as it only contains even powers, we can move the absolute value into the integral to get
\begin{equation} \label{eq:concentration_bound}
\begin{split}
&\phantom{111}\bigg| \int_{\Delta_{k-1} \cap E_{k-1}} \prod_{j=1}^{k-1} \dot{\gamma}_{u_{j}}^{i_{j}} \prod_{j=1}^{k} \sum_{\ell_{j}=0}^{n} \begin{pmatrix} 2n \\ 2 \ell_{j}  \end{pmatrix} (\Delta_{\uu,j} x)^{2 \ell_{j}} (\Delta_{\uu,j} y)^{2n-2 \ell_{j}} d \uu \bigg| \\
&\geq \frac{1}{2^{k}} \bigg( \frac{L}{3} \bigg)^{k-1} \big( 1 - e^{3k \log k -n} \big) \int_{\Delta_{k-1}} \prod_{j=1}^{k} |\Delta_{\uu,j} \gamma|^{2n} d\uu, 
\end{split}
\end{equation}
where we used the lower bound in \eqref{eq:second_part} to reduce the integrand to $\frac{1}{2^{k}} \prod_{j} |\Delta_{\uu,j} \gamma|^{2n}$, the bound \eqref{eq:first_part} to replace $\prod_{j} |\dot{\gamma}_{u_{j}}^{i_{j}}|$ by $(L/3)^{k-1}$, and Proposition \ref{pr:concentration} to enlarge the domain of integration to $\Delta_{k-1}$ with a compensation of the factor $1 - e^{3k \log k - n}$. Note that the exponent here is $3k \log k -n$ instead of $3k \log k-\frac{n}{2}$ since the power of $|\Delta_{\uu,j} \gamma|$ is $2n$. 

For the integration over the domain $\Delta_{k-1} \cap E_{k-1}^{c}$, it follows from $|\dot{\gamma}_{u_{j}}^{i_{j}}| \leq L$, the upper bound in \eqref{eq:second_part} and Proposition \ref{pr:concentration} that it is bounded by
\begin{equation} \label{eq:out_concentration}
e^{3k \log k - n} L^{k-1} \int_{\Delta_{k-1}} \prod_{j=1}^{k} |\Delta_{\uu,j} \gamma|^{2n} d\uu. 
\end{equation}
Combining \eqref{eq:concentration_bound} and \eqref{eq:out_concentration}, we obtain \eqref{eq:non_degenerate} and thus finish the proof of the theorem. 
\end{proof}

\section{Reconstructing the path} \label{sec:reconstruction}

We are now ready to reconstruct the path from its signature. Recall from \eqref{eq:eta} that
\begin{align*}
\eta_{k} = \delta(3 \epsilon_{k}) + \frac{L}{\sqrt{k}}. 
\end{align*}
Our aim is to find the parameters $\rho_{j} \in [0,1]$, $a_{j}^{x}, a_{j}^{y} \in \{\pm 1\}$ and $\tilde{L} \in \RR^{+}$ such that
\begin{equation} \label{eq:aim}
\sup_{1 \leq j \leq k} \sup_{u \in [\frac{j-1}{k}, \frac{j}{k}]} \big| \tilde{L} \big( a_{j}^{x} \rho_{j}, a_{j}^{y} (1-\rho_{j}) \big) - \dot{\gamma}_{u}  \big| < C \eta_{k}
\end{equation}
for all large enough $k$, where $C$ is some constant depending on $\gamma$ only. In what follows, we will recover these parameters in the order of $\rho_{j}$'s, $a_{j}^{x}, a_{j}^{y}$'s, and then finally $\tilde{L}$. We will also give in Theorem \ref{th:main} the precise value of the constant $C$ and a quantitative characterization of how large $k$ needs to be.

\begin{rmk} \label{rm:eta}
As for the magnitude of $\eta_{k}$, if $\dot{\gamma} \in \cC^{\alpha}$ for some $\alpha \in (0,1)$, then
\begin{align*}
\epsilon_{k} \lesssim \sqrt{\frac{\delta(1/k)}{L}} + \frac{1}{\sqrt{k}} \lesssim k^{-\frac{\alpha}{2}}, 
\end{align*}
and consequently, we have
\begin{align*}
\eta_{k} = \delta (3 \epsilon_{k}) + \frac{L}{\sqrt{k}} \lesssim k^{-\frac{\alpha^{2}}{2}}. 
\end{align*}
The readers might wonder why we add the additional term $\frac{L}{\sqrt{k}}$ to the definition of $\eta_{k}$. In fact, as long as $\gamma$ is not a straight line, then there exists $\lambda > 0$ such that
\begin{align*}
\delta (3 \epsilon_{k}) > 3 \lambda \epsilon_{k} \geq \frac{3 \lambda}{\sqrt{k}}. 
\end{align*}
Thus, $\frac{L}{\sqrt{k}} < C \delta(3 \epsilon_{k})$, and adding this additional term $\frac{L}{\sqrt{k}}$ does not change the magnitude of $\eta_{k}$ for large $k$. Since we will need a lower bound for $\eta_{k}$ in Theorem \ref{th:directions} below, this additional term will save us from introducing the new constant $\lambda$. 
\end{rmk}

\subsection{The unsigned directions}

We start with the recovery of unsigned directions $\rho_{j}$'s. At this stage, we only use the quantities $\sS^{2n}_{k}(w, \Bell)$'s which can be obtained from symmetrization of the signature as in \eqref{eq:symmetrized_sig} and have an expression as in \eqref{eq:expression_symmetrized_sig}. Recall from Section \ref{sec:notations} that
\begin{align*}
\Delta_{j} x = x_{j/k} - x_{(j-1)/k}, \quad \Delta_{j} y = y_{j/k} - y_{(j-1)/k}, \quad |\Delta_{j} \gamma| = |\Delta_{j} x| + |\Delta_{j} y|. 
\end{align*}
Since we expect $\rho_{j}$ to be close to the increment of$\gamma|_{[\frac{j-1}{k}, \frac{j}{k}]}$, it is natural to introduce for each $j$ the unique real number $r_{j} \in [0,1]$ such that
\begin{align*}
|\Delta_{j} x|: |\Delta_{j} y| = r_{j}: (1 - r_{j}). 
\end{align*}
Also recall from \eqref{eq:concentration_set} the definition of $E_{k-1}$, which we will be frequently using throughout this section. The following two elementary lemmas will be useful in the sequel.

\begin{lem} \label{le:directions}
For every $k$ such that $\delta(3 \epsilon_{k}) < \frac{L}{2}$ and every $\uu \in \Delta_{k-1} \cap E_{k-1}$, we have
\begin{equation} \label{eq:direction_preliminary}
\sup_{1 \leq j \leq k} \bigg| \frac{|\Delta_{\uu,j} x|}{|\Delta_{\uu,j} \gamma|} - r_{j}  \bigg| < \frac{\eta_{k}}{L}. 
\end{equation}
\end{lem}
\begin{proof}
Fix $\uu \in \Delta_{k-1} \cap E_{k-1}$ and $1 \leq j \leq k$. Let
\begin{align*}
Q_{j} = \big[ \frac{j-1}{k} - \epsilon_{k}, \phantom{1} \frac{j}{k} + \epsilon_{k} \big] \cap [0,1]. 
\end{align*}
Then, $Q_{j} \supset \big( [\frac{j-1}{k}, \frac{j}{k}] \cup [u_{j-1}, u_{j}] \big)$ and $|Q_{j}| < 3 \epsilon_{k}$. It suffices to prove \eqref{eq:direction_preliminary} in the case when $|\dot{x}_{v}| \geq \frac{L}{2}$ for some $v \in Q_{j}$, for otherwise we would have $|\dot{y}_{v}| \geq \frac{L}{2}$ for some $v \in Q_{j}$ and can get \eqref{eq:direction_preliminary} through
\begin{align*}
\bigg| \frac{|\Delta_{\uu,j} x|}{|\Delta_{\uu,j} \gamma|} - r_{j} \bigg|  =  \bigg| \frac{|\Delta_{\uu,j} y|}{|\Delta_{\uu,j} \gamma|} - (1-r_{j}) \bigg|. 
\end{align*}
Since $|\dot{x}_{v}| \geq \frac{L}{2}$ for some $v \in Q_{j}$, $|Q_{j}| < 3 \epsilon_{k}$, and $\delta (3 \epsilon_{k}) < \frac{L}{2}$, we have
\begin{align*}
\inf_{t \in Q_{j}} |\dot{x}_{t}| > 0. 
\end{align*}
In particular, $x$ is monotone in $Q_{j}$ and $\dot{x}$ does not change its sign. If $y$ is also monotone in $Q_{j}$, then $|\Delta_{\uu,j} \gamma| = L(u_{j} - u_{j-1})$ and $|\Delta_{j} \gamma| = \frac{L}{k}$. Thus, there exist $v, \tilde{v} \in Q_{j}$ such that
\begin{align*}
\frac{|\Delta_{\uu,j} x|}{|\Delta_{\uu,j} \gamma|} = \frac{|\dot{x}_{v}|}{L}, \qquad r_{j} = \frac{|\dot{x}_{\tilde{v}}|}{L}, 
\end{align*}
and \eqref{eq:direction_preliminary} follows since $|v-\tilde{v}| < |Q_{j}| < 3 \epsilon_{k}$. If $y$ is not monotone in $Q_{j}$, then there exists $s \in Q_{j}$ such that $\dot{y}_{s} = 0$, and we have the bound
\begin{align*}
\inf_{t \in Q_{j}} |\dot{x}_{t}| \geq L - \delta(3 \epsilon_{k}). 
\end{align*}
This implies that
\begin{align*}
\frac{L - \delta(3\epsilon_{k})}{L} \leq \frac{|\Delta_{\uu,j} x|}{|\Delta_{\uu,j} \gamma|} \leq 1, \qquad \frac{L - \delta(3\epsilon_{k})}{L} \leq r_{j} \leq 1, 
\end{align*}
and the bound \eqref{eq:direction_preliminary} follows immediately. 
\end{proof}

\begin{lem} \label{le:binomial_bound}
	Let $m  \geq 1$. For every $\ell = 0, \dots, m$ and $p \in [0,1]$ we have
	\begin{align*}
	\begin{pmatrix} m \\ \ell \end{pmatrix} p^{\ell} (1-p)^{m-\ell} \leq m e^{-\frac{m x^{2}}{2}}, 
	\end{align*}
	where $x = \frac{\ell}{m} - p  \in [-p, 1-p]$.  
\end{lem}
\begin{proof}
	The lemma is clearly true if $p=0$ or $1$. For $p \in (0,1)$, using the estimate
	\begin{align*}
	\ell \log \ell - (\ell-1) = \int_{1}^{\ell} \log x dx < \log (\ell !) < \int_{1}^{\ell+1} \log x dx = (\ell+1) \log (\ell+1) - \ell
	\end{align*}
	as well as those for $(m-\ell)!$ and $m!$, we obtain the bound
	\begin{align*}
	f(x) := \begin{pmatrix} m \\ \ell \end{pmatrix} p^{\ell} q^{m-\ell} < m \bigg[ \bigg( \frac{p}{p+x} \bigg)^{p+x} \bigg( \frac{q}{q-x} \bigg)^{q-x} \bigg]^{m}, 
	\end{align*}
	where $x = \frac{\ell}{m} - p \in [-p,q]$, and $q = 1-p$. Let
	\begin{align*}
	g(x) = (p+x) \big( \log p - \log(p+x) \big) + (q-x) \big( \log q - \log (q-x) \big). 
	\end{align*}
	It is straightforward to check that
	\begin{align*}
	g'(x) = \log \bigg( 1 - \frac{x}{q(p+x)} \bigg)
	\end{align*}
	satisfies $g'(x) \leq -x$ for $x \in [0,q]$ and $g'(x) \geq |x|$ for $x \in [-p,0]$. As a consequence, we have
	\begin{align*}
	g(x) \leq - \frac{x^{2}}{2}, \qquad x \in [-p,q], 
	\end{align*}
	and that
	\begin{align*}
	f(x) = m e^{m g(x)} \leq m e^{-\frac{m x^{2}}{2}}. 
	\end{align*}
	This finishes the proof. 
\end{proof}

We are now ready to prove the following main theorem about the unsigned directions.

\begin{thm} \label{th:directions}
For any $k \geq 2$ such that $\eta_{k} < \frac{L}{6}$, $n = 4 k^{2} \log k$ and $j=1, \dots, k$, we have
\begin{equation} \label{eq:directions}
\bigg( \sum_{|w|=k-1} \sum_{\Bell: |\frac{\ell_{j}}{n} - r_{j}| \geq \frac{2 \eta_{k}}{L}} \big| \sS^{2n}_{k}(w, \Bell) \big|  \bigg) \bigg/  \bigg( \sum_{|w|=k-1} \sum_{\Bell \in \lL^{n}_{k}} \big| \sS^{2n}_{k}(w, \Bell) \big| \bigg) < 16 k^{5} \bigg( \frac{12}{k^{4}} \bigg)^{k}, 
\end{equation}
where the sum for $w$ is taken over all words with length $k-1$, and in the numerator, the sum for multi-indices $\Bell$ is taken over all $\Bell \in \lL^{n}_{k}$ such that $|\frac{\ell_{j}}{n} - r_{j}| \geq \frac{2 \eta_{k}}{L}$. 
\end{thm}
\begin{proof}
	Since both the numerator and denominator on the left hand side of \eqref{eq:directions} scale like $L^{2n+k-1}$, we can assume without loss of generality that $L = 1$. Now fix $n$ and large $k$ whose values will be chosen later, and also fix $1 \leq j \leq k$. Since $|\dot{\gamma}| \equiv 1$, for any word $w$ with length $k-1$, we have
	\begin{equation} \label{eq:numerator_direction_expression}
	\begin{split}
	&\phantom{1} \sum_{\Bell: |\frac{\ell_{j}}{n} - r_{j}| > 2 \eta_{k}} \big| \sS^{2n}_{k}(w, \Bell) \big| \leq e^{3 k \log k - n} \int_{\Delta_{k-1}} \prod_{i=1}^{k} |\Delta_{\uu,i} \gamma|^{2n} d \uu \\
	&+ \int_{\Delta_{k-1} \cap E_{k-1}} \sum_{|\frac{\ell_{j}}{n} - r_{j}| > 2 \eta_{k}} \begin{pmatrix} 2n \\ 2 \ell_{j} \end{pmatrix} \bigg( \frac{|\Delta_{\uu,j} x|}{|\Delta_{\uu,j} \gamma|} \bigg)^{2 \ell_{j}} \bigg( \frac{|\Delta_{\uu,j} y|}{|\Delta_{\uu,j} \gamma|} \bigg)^{2n - 2\ell_{j}} \prod_{i=1}^{k} |\Delta_{\uu,i} \gamma|^{2n} d \uu, 
	\end{split}
	\end{equation}
	where we have split the integral into two disjoint domains and applied Proposition \ref{pr:concentration} to bound the one on $\Delta_{k-1} \cap E_{k-1}^{c}$. The range of sum is $|\frac{\ell_{j}}{n} - r_{j}| > 2 \eta_{k}$ since we assumed $L=1$. 
	
	Now we need to bound the integrand for the second term on the right hand side of \eqref{eq:numerator_direction_expression}. Fix $\uu \in \Delta_{k-1} \cap E_{k-1}$, and let
	\begin{align*}
	p = \frac{|\Delta_{\uu,j} x|}{|\Delta_{\uu,j} \gamma|}. 
	\end{align*}
	Lemma \ref{le:directions} implies that if $|\frac{\ell_{j}}{n} - r_{j}| > 2 \eta_{k}$, then we must have $|\frac{\ell_{j}}{n} - p| > \eta_{k}$. 
	Thus, using Lemma \ref{le:binomial_bound}, we obtain the pointwise bound
	\begin{align*}
	\begin{pmatrix} 2n \\ 2 \ell_{j} \end{pmatrix} p^{2 \ell_{j}} (1-p)^{2n - 2\ell_{j}} \leq 2n \exp \big(-n \big|\frac{\ell_{j}}{n}-p\big|^{2} \big) < 2n e^{- n \eta_{k}^{2}}. 
	\end{align*}
	Since this bound holds for all $\uu \in \Delta_{k-1} \cap E_{k-1}$ (recall that $p$ depends on $\uu$) and all $\ell_{j}$ such that $\big| \frac{\ell_{j}}{n} - r_{j} \big| > 2 \eta_{k}$, we get
	\begin{align*}
	\sum_{|\frac{\ell_{j}}{n} - r_{j}| > 2 \eta_{k}} \begin{pmatrix} 2n \\ 2 \ell_{j} \end{pmatrix} \bigg( \frac{|\Delta_{\uu,j} x|}{|\Delta_{\uu,j} \gamma|} \bigg)^{2 \ell_{j}} \bigg( \frac{|\Delta_{\uu,j} y|}{|\Delta_{\uu,j} \gamma|} \bigg)^{2n - 2\ell_{j}} \leq 2n^{2} e^{-n \eta_{k}^{2}}, \quad \forall \uu \in \Delta_{k-1} \cap E_{k-1}. 
	\end{align*}
	Now, plugging the above pointwise bound into \eqref{eq:numerator_direction_expression} and summing over all words $w$ with length $k-1$, we obtain an upper bound for the numerator in \eqref{eq:numerator_direction_expression} as
	\begin{equation} \label{eq:numerator_direction}
	\sum_{w} \sum_{\Bell: |\frac{\ell_{j}}{n} - r_{j}| \geq \frac{2 \eta_{k}}{L}} \big| \sS^{2n}_{k}(w, \Bell) \big| < C_{n,k} \int_{\Delta_{k-1}} \prod_{i=1}^{k} |\Delta_{\uu,i} \gamma|^{2n} d\uu, 
	\end{equation}
	where
	\begin{equation} \label{eq:numerator_constant}
	C_{n,k} = 2^{k-1} \big( e^{3k \log k - n} + 2n^{2} e^{- n \eta_{k}^{2}} \big). 
	\end{equation}
	Applying Theorem \ref{th:concentration}, we have a lower bound for the denominator $\sum_{w} \sum_{\Bell} |\sS^{2n}_{k}(w, \Bell)|$ in \eqref{eq:directions}, which, when combined with \eqref{eq:numerator_direction} and \eqref{eq:numerator_constant}, implies that the left hand side of \eqref{eq:directions} is bounded by
	\begin{align*}
	\frac{2^{k-1} \big( e^{3k \log k - n} + 2n^{2} e^{-n \eta_{k}^{2}} \big)}{\frac{3}{6^{k}} - \big(1 + \frac{3}{6^{k}} \big) e^{3k \log k - n}}. 
	\end{align*}
	The theorem then follows by taking $n = 4 k^{2} \log k$ and noting that $\eta_{k}^{2} \geq \frac{1}{k}$ (recall we have set $L=1$).  
\end{proof}

\begin{rmk}
	Note that the left hand side of \eqref{eq:directions} is always smaller than $1$, so the theorem is meaningful when the right hand side also falls below $1$, which is the case when $k \geq 4$ (in fact it is already smaller than $\frac{1}{8}$ when $k=4$). 
\end{rmk}

The following easy corollary enables one to select the directions $\rho_{j}$ (up to the sign) for our piecewise linear approximation.

\begin{cor} \label{cor:finding_directions}
[\textbf{Choosing directions}] Let $k \geq 4$ be such that $\eta_{k} < \frac{L}{6}$. For each $0 \leq j \leq k$, there exists $\rho_{j} \in [0,1]$ such that
\begin{equation} \label{eq:finding_directions}
\bigg(\sum_{|w|=k-1} \sum_{\Bell: |\frac{\ell_{j}}{n} - \rho_{j}| \leq \frac{2 \eta_{k}}{L}} \big| \sS^{2n}_{k}(w, \Bell) \big| \bigg) \bigg/  \bigg( \sum_{|w|=k-1} \sum_{\Bell \in \lL^{n}_{k}} \big| \sS^{2n}_{k}(w, \Bell) \big| \bigg)  > \frac{1}{2}, 
\end{equation}
where in the numerator, the sum over $\Bell$ in the appropriate subset of $\lL^{n_{k}}$ is implicit. Moreover, if $\{\rho_{j}\}_{j=1}^{k}$ is any set that satisfies \eqref{eq:finding_directions}, then we must have
\begin{align*}
|\rho_{j} - r_{j}| < \frac{4 \eta_{k}}{L}
\end{align*}
for all $j = 1, \cdots, k$. 
\end{cor}
\begin{proof}
The existence of the set $\{\rho_{j}\}$ that satisfies \eqref{eq:finding_directions} follows directly by setting $\rho_{j} = r_{j}$ and applying Theorem \ref{th:directions}. Conversely, if $|\rho_{j} - r_{j}| \geq \frac{4 \eta_{k}}{L}$ for some $j$, then $|\frac{\ell_{j}}{n} - \rho_{j}| < \frac{2 \eta_{k}}{L}$ implies that
\begin{align*}
|\frac{\ell_{j}}{n} - r_{j}| > \frac{2 \eta_{k}}{L}. 
\end{align*}
By Theorem \ref{th:directions}, this set of $\{\rho_{j}\}$ must violate \eqref{eq:finding_directions}. This completes the proof. 
\end{proof}

\begin{rmk}
Note that on the left hand side of \eqref{eq:finding_directions}, there is only information that is available from the signature of $\gamma$. If we choose $\{\rho_{j}\}$ according to \eqref{eq:finding_directions}, then Corollary \ref{cor:finding_directions} guarantees that these unsigned directions we recover from the signature must be close to the true directions $\{r_{j}\}$. 
\end{rmk}

\begin{rmk} \label{rm:prior_information}
The readers might have noticed that the criterion of choosing the $\rho_{j}$'s above involves the knowledge of $\eta_{k}$. This is of course not a problem if we know the modulus of continuity of $\dot{\gamma}$ in advance, as in this case
\begin{align*}
\frac{\eta_{k}}{L} = \frac{\delta(3 \epsilon_{k})}{L} + \frac{1}{\sqrt{k}}
\end{align*}
is explicitly known. But even if that information is not available, we can always choose a sequence $\alpha_{k}$ which decreases to $0$ slowly enough such that $\alpha_{k} > \frac{2 \eta_{k}}{L}$, and replace the range of the sum in the numerator of \eqref{eq:finding_directions} by
\begin{align*}
\Bell: |\frac{\ell_{j}}{n} - \rho_{j}| \leq \alpha_{k}. 
\end{align*}
The directions $\{\rho_{j}\}$ chosen in this way satisfy
\begin{align*}
|\rho_{j} - r_{j}| < \alpha_{k} + \frac{2 \eta_{k}}{L}, 
\end{align*}
which still goes to $0$ as $k \rightarrow +\infty$, but at a slower rate than $\eta_{k}$. 
\end{rmk}

\subsection{The signs}

We now turn to the recovery of the sign of the direction of each piece. For this, we need to move one level up in the signatures, which requires more information than the $\sS^{2n}_{k}(w, \Bell)$'s. 

We look at the signature at level $2nk + k$ (in addition to the level $2nk+k-1$ before), and divide it into $k$ blocks of size $2n$ except one of them which has size $2n+1$, still with one letter separating consecutive blocks.  More precisely, for any word $w = e_{i_{1}} \cdots e_{i_{k-1}}$, any multi-index $\Bell \in \lL^{n}_{k}$, and any $1 \leq i \leq k$, we let $\wW^{2n}_{k,i,x}(w, \Bell)$ denote the set of words
\begin{align*}
w' = w_{1} * e_{i_{1}} * \cdots * e_{i_{k-1}} * w_{k}
\end{align*}
such that $|w_{j}|_{y} = 2n-2\ell_{j}$ for every $1 \leq j \leq k$, $|w_{j}|_{x} = 2 \ell_{j}$ for every $j \neq i$, but $|w_{i}|_{x} = 2\ell_{i} + 1$. The set $\wW_{k,i,x}^{2n}(w, \Bell)$ is different from $\wW_{k}^{2n}(w, \Bell)$ defined in \eqref{eq:blocks_word} in that the $i$-th block has size $2n+1$ instead of $2n$, and contains $2 \ell_{i} + 1$ $x$'s instead of $2 \ell_{i}$. We define the set $\wW^{2n}_{k,i,y}(w, \Bell)$ in the same way except that $|w_{i}|_{x} = 2 \ell_{i}$ and $|w_{i}|_{y} = 2n - 2 \ell_{i} + 1$. 

We then define the quantities $\sS^{2n}_{k,i,x}(w, \Bell)$ and $\sS^{2n}_{k,i,y}(w,\Bell)$ to be
\begin{equation} \label{eq:symmetrized_sig_sign}
\begin{split}
\sS^{2n}_{k,i,x}(w, \Bell) = (2n+1) \big( (2n)! \big)^{k} \sum_{w' \in \wW^{2n}_{k,i,x}(w, \Bell)} C(w'), \\ \sS^{2n}_{k,i,y}(w, \Bell) = (2n+1) \big( (2n)! \big)^{k} \sum_{w' \in \wW^{2n}_{k,i,y}(w, \Bell)} C(w'). 
\end{split}
\end{equation}
The aim of introducing these quantities is to recover the sign of $x$ and $y$ directions in the $i$-th piece of the path via comparison with $\sS^{2n}_{k}(w, \Bell)$. Similar to \eqref{eq:expression_symmetrized_sig}, we have the expressions
\begin{equation} \label{eq:expression_sign}
\begin{split}
\sS^{2n}_{k,i,x}(w, \Bell) = \int_{\Delta_{k-1}} &\prod_{j=1}^{k-1} \dot{\gamma}_{u_{j}}^{i_{j}} \cdot \begin{pmatrix} 2n+1 \\ 2\ell_{i} + 1 \end{pmatrix} (\Delta_{\uu,i}x)^{2 \ell_{i}+1} (\Delta_{\uu,i}y)^{2n-2\ell_{i}} \\
&\prod_{j \neq i} \begin{pmatrix} 2n \\ 2\ell_{j}  \end{pmatrix} (\Delta_{\uu,j}x)^{2\ell_{j}} (\Delta_{\uu,j}y)^{2n-2\ell_{j}} d\uu, 
\end{split}
\end{equation}
and
\begin{equation} \label{eq:expression_sign_y}
\begin{split}
\sS^{2n}_{k,i,y}(w, \Bell) = \int_{\Delta_{k-1}} &\prod_{j=1}^{k-1} \dot{\gamma}_{u_{j}}^{i_{j}} \cdot \begin{pmatrix} 2n+1 \\ 2\ell_{i} \end{pmatrix} (\Delta_{\uu,i}x)^{2\ell_{i}} (\Delta_{\uu,i}y)^{2n + 1 - 2\ell_{i}} \\
&\prod_{j \neq i} \begin{pmatrix} 2n \\2\ell_{j}  \end{pmatrix} (\Delta_{\uu,j}x)^{2\ell_{j}} (\Delta_{\uu,j}y)^{2n-2\ell_{j}} d\uu. 
\end{split}
\end{equation}
Again, we emphasize that \eqref{eq:symmetrized_sig_sign} is the information available from the signature, while \eqref{eq:expression_sign} and \eqref{eq:expression_sign_y} are expressions for these quantities. 

Unlike in choosing unsigned directions where we sum up the $\sS^{2n}_{k}(w, \Bell)$'s over all words $w$ that have length $k-1$ (see Corollary \ref{cor:finding_directions}), in what follows, we will use the quantities \eqref{eq:symmetrized_sig_sign} with only one particular word $w^{*}$. We choose this word as follows. Let $\{\rho_{j}\}$ be the set of unsigned directions chosen according to Corollary \ref{cor:finding_directions}, then for each $j = 1, \cdots, k-1$, we let
\begin{equation} \label{eq:choice_letter}
\begin{split}
e_{i_{j}} = x, \qquad \text{if} \phantom{1} \rho_{j} \geq \frac{1}{2}, \\
e_{i_{j}} = y, \qquad \text{if} \phantom{1} \rho_{j} < \frac{1}{2}, 
\end{split}
\end{equation}
and we set the word $w^{*}$ to be
\begin{equation} \label{eq:choice_word}
w^{*} = e_{i_{1}}  \cdots  e_{i_{k-1}}. 
\end{equation}
The word $w^{*}$ chosen above guarantees that $\prod_{j} \dot{\gamma}_{u_{j}}^{i_{j}}$ is bounded away from $0$ in the region $\Delta_{k-1} \cap E_{k-1}$. This is the content of the following proposition. 

\begin{prop} \label{pr:word}
	Let $k$ be large enough such that $\eta_{k} < \frac{L}{32}$, and let $w^{*}$ be the word chosen according to \eqref{eq:choice_letter} and \eqref{eq:choice_word}. Then, for all $\uu \in E_{k-1}$ and all $j = 1, \dots, k-1$, we have
	\begin{align*}
	|\dot{\gamma}_{u_{j}}^{i_{j}}| > \frac{L}{3}
	\end{align*}
	and that $\dot{\gamma}_{u_{j}}^{i_{j}}$ does not change its sign in the domain $E_{k-1}$. 
\end{prop}
\begin{proof}
	Fix $k$ as in the assumption and $1 \leq j \leq k-1$. If $\rho_{j} \geq \frac{1}{2}$, then  $e_{i_{j}} = x$, and Corollary \ref{cor:finding_directions} implies
	\begin{align*}
	r_{j} \geq \frac{1}{2} - \frac{4 \eta_{k}}{L} > \frac{3}{8}. 
	\end{align*}
	Thus, there exist $v, \tilde{v} \in [\frac{j-1}{k}, \frac{j}{k}]$ such that $|\dot{x}_{v}| > \frac{3}{8} (|\dot{x}_{v}| + |\dot{y}_{\tilde{v}}|)$. Using $|\dot{y}_{\tilde{v}}| = L - |\dot{x}_{\tilde{v}}|$, we see that either $|\dot{x}_{v}|$ or $|\dot{x}_{\tilde{v}}|$ is bigger than $\frac{3L}{8}$. If $\uu \in E_{k-1}$, then both $|u-v|$ and $|u-\tilde{v}|$ is smaller than $2 \epsilon_{k}$, so it follows that (recall $\eta_{k} > \delta (3 \epsilon_{k})$)
	\begin{align*}
	|\dot{x}_{u_{j}}| > \frac{3L}{8} - \eta_{k} > \frac{3L}{8} - \frac{L}{32} > \frac{L}{3}. 
	\end{align*}
	Similarly, we have $|\dot{y}_{u_{j}}| > \frac{L}{3}$ for all $\uu \in E_{k-1}$ if $\rho_{j} < \frac{1}{2}$. In particular, the continuity of the derivatives ensures these $\dot{x}_{u_{j}}$'s and $\dot{y}_{u_{j}}$'s do not change their signs in the domain $E_{k-1}$. 
\end{proof}

Note that the word $w^{*}$ chosen above has all the properties we used in Theorem \ref{th:concentration}. Theorem \ref{th:concentration} only gives the existence of such a word, but here we choose it explicitly based on the recovery of the unsigned directions. We now determine the signs of the $i$-th piece as follows.

\begin{defn} \label{de:finding_signs}
[\textbf{Determining the signs}] Fix $k \geq 4$ such that $\delta (\frac{1}{k}) < \frac{L}{6}$, and let $n = 4 k^{2} \log k$. Let $w^{*}$ be the word chosen according to \eqref{eq:choice_letter} and \eqref{eq:choice_word}. We then choose the sign $a_{i}^{x}, a_{i}^{y} \in \{\pm 1\}$ for the $i$-th linear piece by: 
\begin{align*}
a_{i}^{x} = 1, \qquad \text{if} \phantom{11} \frac{\sum_{\Bell \in \lL^{n}_{k}} \sS^{2n}_{k}(w^{*}, \Bell)}{\sum_{\Bell \in \lL^{n}_{k}} \sS^{2n}_{k,i,x}(w^{*}, \Bell)} \geq 0, \\
a_{i}^{x} = -1, \qquad \text{if} \phantom{1} \frac{\sum_{\Bell \in \lL^{n}_{k}} \sS^{2n}_{k}(w^{*}, \Bell)}{\sum_{\Bell \in \lL^{n}_{k}} \sS^{2n}_{k,i,x}(w^{*}, \Bell)} < 0. 
\end{align*}
The choice for $a_{i}^{y}$ is the same except that one replaces $\sS^{2n}_{k,i,x}(w^{*}, \Bell)$ by $\sS^{2n}_{k,i,y}(w^{*}, \Bell)$. 
\end{defn}

It appears that the above choices of the signs depend on $k$, and the choices may be different if $k$ changes. But it turns out that the choices above remain stable for all sufficiently large $k$, and they indeed give the correct signs as long as the directions are not close to degenerate. The remaining of this subsection will be devoted to the verification of this stability.

\begin{thm} \label{th:signs}
Let $k \geq 4$ be sufficiently large such that $\eta_{k} < \frac{L}{32}$, and $n = 4 k^{2} \log k$. If $r_{i} \geq \frac{2 \eta_{k}}{L}$, then
\begin{align*}
\frac{\sum_{\Bell} \sS^{2n}_{k}(w^{*}, \Bell)}{\sum_{\Bell} \sS^{2n}_{k,i,x}(w^{*}, \Bell)} \geq \frac{1}{6 \epsilon_{k} L} \qquad \text{if} \phantom{1} \Delta_{i} x > 0, \\
\frac{\sum_{\Bell} \sS^{2n}_{k}(w^{*}, \Bell)}{\sum_{\Bell} \sS^{2n}_{k,i,x}(w^{*}, \Bell)} \leq - \frac{1}{6 \epsilon_{k} L} \qquad \text{if} \phantom{1} \Delta_{i} x < 0. 
\end{align*}
Similarly, if $r_{i} \leq 1 - \frac{2 \eta_{k}}{L}$, then
\begin{align*}
\frac{\sum_{\Bell} \sS^{2n}_{k}(w^{*}, \Bell)}{\sum_{\Bell} \sS^{2n}_{k,i,y}(w^{*}, \Bell)} \geq \frac{1}{6 \epsilon_{k} L} \qquad \text{if} \phantom{1} \Delta_{i} y > 0, \\
\frac{\sum_{\Bell} \sS^{2n}_{k}(w^{*}, \Bell)}{\sum_{\Bell} \sS^{2n}_{k,i,y}(w^{*}, \Bell)} < - \frac{1}{6 \epsilon_{k} L} \qquad \text{if} \phantom{1} \Delta_{i} y < 0. 
\end{align*}
All the sums above are taken over $\Bell \in \lL^{n}_{k}$, and $\Delta_{i} x$ and $\Delta_{i} y$ are defined in \eqref{eq:increments_standard}. 
\end{thm}

\begin{rmk}
The above theorem guarantees that as long as the $i$-th piece of $\gamma$ is not too horizontal or vertical (corresponding to the assumptions $r_{i} \geq \frac{2 \eta_{k}}{L}$ and $r_{i} \leq 1 - \frac{2 \eta_{k}}{L}$), then the choices in Definition \ref{de:finding_signs} do give the correct signs for all sufficiently large $k$. The case $r_{i} < \frac{2 \eta_{k}}{L}$ and $r_{i} > 1 - \frac{2 \eta_{k}}{L}$ are not covered, but since the $i$-th piece would be almost horizontal or vertical in that situation, the choice of the sign would not affect accuracy. 
\end{rmk}

\begin{proof}
(of Theorem \ref{th:signs}). We only prove the first case when $r_{i} \geq \frac{2 \eta_{k}}{L}$ and $\Delta_{i} x > 0$; the other three cases are essentially the same. By the expressions \eqref{eq:expression_symmetrized_sig} and \eqref{eq:expression_sign}, we can write
\begin{align*}
\sum_{\Bell \in \lL^{n}_{k}} \sS_{k}^{2n}(w^{*}, \Bell) = \int_{\Delta_{k-1}} \nN(\uu) d \uu, \qquad  \sum_{\Bell \in \lL^{n}_{k}} \sS_{k,i,x}^{2n}(w^{*}, \Bell) = \int_{\Delta_{k-1}} \dD(\uu) d \uu. 
\end{align*}
Here, $\nN(\uu)$ and $\dD(\uu)$ are respectively given by
\begin{align*}
\nN(\uu) = \frac{1}{2^{k}} \prod_{j=1}^{k-1} \dot{\gamma}_{u_{j}}^{i_{j}} \prod_{j=1}^{k} \bigg( \big(\Delta_{\uu,j} x + \Delta_{\uu,j} y \big)^{2n} + \big( \Delta_{\uu,j} x - \Delta_{\uu,j} y \big)^{2n} \bigg), 
\end{align*}
and
\begin{equation} \label{eq:expression_denominator}
\dD(\uu) = \frac{1}{2^{k}} \prod_{j=1}^{k-1} \dot{\gamma}_{u_{j}}^{i_{j}} \prod_{j=1}^{k} \bigg( \big(\Delta_{\uu,j} x + \Delta_{\uu,j} y \big)^{2n + \delta_{i,j}} + \big( \Delta_{\uu,j} x - \Delta_{\uu,j} y \big)^{2n + \delta_{i,j}} \bigg), 
\end{equation}
where $\delta_{i,j} = 1$ if $j=i$ and $0$ otherwise. We now need to estimate the ratio of two integrals, both over the domain $\Delta_{k-1}$. Similar as before, we decompose both integrals into domains $\Delta_{k-1} \cap E_{k-1}$ and $\Delta_{k-1} \cap E_{k-1}^{c}$. We postpone the estimate of the latter one to the next lemma, and first consider the quantity
\begin{equation} \label{eq:integral_ratio_concentration}
\frac{\int_{\Delta_{k-1} \cap E_{k-1}} \nN(u) d \uu}{\int_{\Delta_{k-1} \cap E_{k-1}} \dD(u) d \uu}. 
\end{equation}
We hope to control the ratio of the two integrals by means of a pointwise bound on
\begin{equation} \label{eq:pointwise_ratio}
\frac{\nN(\uu)}{\dD(\uu)} = \frac{\big(\Delta_{\uu,i} x + \Delta_{\uu,i} y \big)^{2n} + \big( \Delta_{\uu,i} x - \Delta_{\uu,i} y \big)^{2n}}{\big(\Delta_{\uu,i} x + \Delta_{\uu,i} y \big)^{2n+1} + \big( \Delta_{\uu,i} x - \Delta_{\uu,i} y \big)^{2n+1}}, \quad \uu \in \Delta_{k-1} \cap E_{k-1}. 
\end{equation}
For this, we first note that similar as in Proposition \ref{pr:word}, the assumption $r_{i} \geq \frac{2 \eta_{k}}{L}$ implies that $|\dot{x}_{v}| \geq 2 \eta_{k}$ for some $v \in [\frac{i-1}{k}, \frac{i}{k}]$. Also, since $\Delta_{i} x > 0$ and $2 \eta_{k} > \delta(\frac{1}{k})$, we actually have $\dot{x}_{v} \geq 2 \eta_{k}$. If $\uu \in E_{k-1}$, we will have $|u_{i} - v| < 2 \epsilon_{k}$ and hence
\begin{equation} \label{eq:x_derivative_positive}
\dot{x}_{u_{i}} > 2 \eta_{k} - \delta(2 \epsilon_{k}) > \eta_{k}. 
\end{equation}
In particular, this implies $\Delta_{\uu,i} x$ is positive and so is the ratio $\frac{\nN(u)}{\dD(u)}$ for all $\uu \in E_{k-1}$. Since
\begin{align*}
\max \big\{ |\Delta_{\uu,i} x + \Delta_{\uu,i} y|, |\Delta_{\uu,i} x - \Delta_{\uu,i} y| \big\} = |\Delta_{\uu,i} \gamma|, 
\end{align*}
we then have the bound
\begin{align*}
\frac{\nN(\uu)}{\dD(\uu)} \geq \frac{1}{|\Delta_{\uu,i} \gamma|} \geq \frac{1}{3 \epsilon_{k} L}
\end{align*}
for all $\uu \in E_{k-1}$. Note that the choice of the word $w^{*}$ guarantees both $\nN(\uu)$ and $\dD(\uu)$ keep their signs unchanged in the domain $E_{k-1}$, so the pointwise bound carries to the ratio of the integrals, which gives
\begin{equation} \label{eq:sign_concentration}
\frac{\int_{\Delta_{k-1} \cap E_{k-1}} \nN(\uu) d \uu}{\int_{\Delta_{k-1} \cap E_{k-1}} \dD(\uu) d \uu} \geq \frac{1}{3 \epsilon_{k} L}. 
\end{equation}
The proof of the theorem will be complete by combining \eqref{eq:sign_concentration} Lemma \ref{le:control_ratio} below. 
\end{proof}

We now give the lemma which allows us to replace the ratio of the integrals $\frac{\int_{\Delta_{k-1}} \nN(\uu) d \uu}{\int_{\Delta_{k-1}} \dD(\uu) d \uu}$ by the integrations over the sub-domain $\Delta_{k-1} \cap E_{k-1}$ as in \eqref{eq:integral_ratio_concentration}. 

\begin{lem} \label{le:control_ratio}
Let $\nN(\uu)$ and $\dD(\uu)$ be as given above. Then, we have
\begin{equation} \label{eq:concentration_sign}
\frac{ \big| \int_{\Delta_{k-1} \cap E_{k-1}^{c}} \nN(\uu) d\uu \big|}{ \big| \int_{\Delta_{k-1} \cap E_{k-1}} \nN(\uu) d\uu \big|} < \frac{1}{3}, \qquad \frac{ \big| \int_{\Delta_{k-1} \cap E_{k-1}^{c}} \dD(\uu) d\uu \big|}{\big| \int_{\Delta_{k-1} \cap E_{k-1}} \dD(\uu) d\uu \big|} < \frac{1}{3}. 
\end{equation}
\end{lem}
\begin{proof}
Since the word $w^{*}$ chosen above has all the properties we used in Theorem \ref{th:concentration}, the first inequality in \eqref{eq:concentration_sign} follows as a direct consequence of Proposition \ref{pr:concentration} and Theorem \ref{th:concentration}. For the second inequality, since it involves one more power of $|\Delta_{\uu,i} \gamma|$, we first need a modified version of Proposition \ref{pr:concentration}. 

Let $\eE_{k-1}$ be the same set as defined in Proposition \ref{pr:concentration}, so we have
\begin{align*}
\prod_{j=1}^{k} |\Delta_{\uu,j} \gamma|^{2n} < e^{-n} \prod_{j=1}^{k} |\Delta_{\vv,j}|^{2n}, \qquad \forall \uu \in \Delta_{k-1} \cap E_{k-1}^{c}, \phantom{1} \vv \in \eE_{k-1}. 
\end{align*}
On the other hand, similar as in Lemma \ref{le:standard_location}, if $\delta(\frac{2}{k}) < \frac{L}{2}$ and $\vv \in \eE_{k-1}$, we will have
\begin{align*}
|\Delta_{\vv,i} \gamma| \geq \frac{L - \delta(\frac{2}{k})}{2k} \geq \frac{L}{4k} \geq \frac{|\Delta_{\uu,i} \gamma|}{4k}. 
\end{align*}
This then implies
\begin{equation} \label{eq:concentration_modified}
|\Delta_{\uu,i} \gamma| \prod_{j=1}^{k} |\Delta_{\uu,j} \gamma|^{2n} < 4k e^{-n} |\Delta_{\vv,i} \gamma| \prod_{j=1}^{k} |\Delta_{\vv,j} \gamma|^{2n}
\end{equation}
for all $\uu \in \Delta_{k-1} \cap E_{k-1}^{c}$ and $\vv \in \eE_{k-1}$. We now arrive at the same situation as in Proposition \ref{pr:concentration} except that there is one more factor of $4k$ on the right hand side. Integrating both sides above in their respective domains and then enlarging $\eE_{k-1}$ to $\Delta_{k-1}$, we get
\begin{equation} \label{eq:sign_negligible}
\int_{\Delta_{k-1} \cap E_{k-1}^{c}} \prod_{j=1}^{k} |\Delta_{\uu,j} \gamma|^{2n+\delta_{i,j}} d \uu < 4k e^{3k \log k -n} \int_{\Delta_{k-1}} \prod_{j=1}^{k} |\Delta_{\uu,j} \gamma|^{2n+\delta_{i,j}} d \uu, 
\end{equation}
where $\delta_{i,j} = 1$ if $j=i$ and is $0$ otherwise. We are now ready to prove the second inequality in \eqref{eq:concentration_sign}. By the expression \eqref{eq:expression_denominator}, we have the pointwise bound
\begin{align*}
	|\dD(\uu)| \leq L^{k-1} \prod_{j=1}^{k} |\Delta_{\uu,j} \gamma|^{2n + \delta_{i,j}}, \qquad \forall \uu \in \Delta_{k-1}, 
\end{align*}
so it follows from \eqref{eq:sign_negligible} that
\begin{equation} \label{eq:denominator_small}
\int_{\Delta_{k-1} \cap E_{k-1}^{c}} \dD(\uu) d \uu < 4k L^{k-1} e^{3k \log k -n} \int_{\Delta_{k-1}} \prod_{j=1}^{k} |\Delta_{\uu,j} \gamma|^{2n + \delta_{i,j}} d \uu. 
\end{equation}
As for the integration of $\dD(\uu)$ in $\Delta_{k-1} \cap E_{k-1}$, we first note from \eqref{eq:x_derivative_positive} that $\Delta_{\uu,i} x$ is positive with the lower bound
\begin{align*}
\Delta_{\uu,i} x \geq \eta_{k} (u_{i} - u_{i-1}) \geq \frac{\eta_{k}}{L} |\Delta_{\uu,i} \gamma|, 
\end{align*}
and hence
\begin{align*}
(\Delta_{\uu,i} x + \Delta_{\uu,i}y)^{2n+1} + (\Delta_{\uu,i} x - \Delta_{\uu,i}y)^{2n+1} \geq (\Delta_{\uu,i} x) |\Delta_{\uu,i} \gamma|^{2n} \geq \frac{\eta_{k}}{L} |\Delta_{\uu,i} \gamma|^{2n+1}. 
\end{align*}
Thus, by the choice of $w^{*}$ and Proposition \ref{pr:word}, we have
\begin{align*}
|\dD (\uu)| \geq \frac{L^{k-1}}{6^{k} \sqrt{k}} \prod_{j=1}^{k} |\Delta_{\uu,j} \gamma|^{2n + \delta_{i,j}}, \quad \forall \uu \in \Delta_{k-1} \cap E_{k-1}, 
\end{align*}
where we have used $\eta_{k} \geq \frac{L}{\sqrt{k}}$. Also, since the choice of $w^{*}$ and the positivity of $\Delta_{\uu,i} x$ guarantee that $\dD(\uu)$ does not change its sign in the domain $\Delta_{k-1} \cap E_{k-1}$, we can change the order of the absolute value sign and integration to get
\begin{align*}
\bigg| \int_{\Delta_{k-1} \cap E_{k-1}} \dD(\uu) d \uu \bigg| \geq \frac{L^{k-1}}{6^{k} \sqrt{k}} \int_{\Delta_{k-1} \cap E_{k-1}} \prod_{j=1}^{k} |\Delta_{\uu,j} \gamma|^{2n + \delta_{i,j}} d \uu. 
\end{align*}
Applying \eqref{eq:sign_negligible}, we can enlarge the domain of integration on the right hand side above to $\Delta_{k-1}$ and obtain
\begin{equation} \label{eq:denominator_big}
\bigg| \int_{\Delta_{k-1} \cap E_{k-1}} \dD(\uu) d \uu \bigg| \geq \frac{L^{k-1}}{6^{k} \sqrt{k}} \big( 1 - 4k e^{3k \log k - n} \big) \int_{\Delta_{k-1}} \prod_{j=1}^{k} |\Delta_{\uu,j} \gamma|^{2n+\delta_{i,j}} d \uu. 
\end{equation}
Combining \eqref{eq:denominator_small} and \eqref{eq:denominator_big}, and taking $n = 4 k^{2} \log k$, we can conclude the proof of this lemma as well as Theorem \ref{th:signs}. 
\end{proof}

So far, we have recovered the unsigned directions $\{\rho_{j}\}$ as well the signs $a_{j}^{x}$, $a_{j}^{y}$'s. This already gives us a unit-length piecewise linear path
\begin{equation} \label{eq:unit_approximation}
\zeta = \frac{1}{k} \big( \theta_{1} * \cdots * \theta_{k} \big), 
\end{equation}
where $\theta_{j} = \big( a_{j}^{x} \rho_{j}, a_{j}^{y} (1 - \rho_{j}) \big)$ for every $j = 1, \dots, k$. Note that we have an abuse use of notation here as the $j$-th piece of $\theta$ above should really be the \textit{path} $\theta_{j} t$ for $t \in [\frac{j-1}{k}, \frac{j}{k}]$ instead of the \textit{direction} $\frac{1}{k} \theta_{j}$. But we choose to stay with \eqref{eq:unit_approximation} for notational simplicity. We now end this subsection by showing that the piecewise linear $\zeta$ is close to $\gamma$ when the latter is normalized to have length $1$. 

\begin{cor} \label{cor:accuracy_signed_directions}
	Fix an integer $k \geq 4$ such that $\eta_{k} < \frac{L}{32}$. For each $j=1, \dots, k$, let $\rho_{j}$, $a_{j}^{x}$ and $a_{j}^{y}$ be obtained as in Corollary \ref{cor:finding_directions} and Definition \ref{de:finding_signs}. For each $j=1, \cdots, k$, let 
	\begin{align*}
	\theta_{j} = \big( a_{j}^{x} \rho_{j}, a_{j}^{y} (1-\rho_{j}) \big). 
	\end{align*}
	Then, the signed directions $\theta_{j}$'s satisfy
	\begin{equation} \label{eq:accuracy_signed_direction}
	\bigg| \theta_{j} - \bigg( \frac{\Delta_{j} x}{|\Delta_{j} \gamma|}, \frac{\Delta_{j} y}{|\Delta_{j} \gamma|} \bigg) \bigg| < \frac{12 \eta_{k}}{L}
	\end{equation}
	for every $j = 1, \dots, k$. As a consequence, we have
	\begin{align*}
	\sup_{1 \leq j \leq k} \sup_{u \in [\frac{j-1}{k}, \frac{j}{k}]} \big| \theta_{j} - \frac{\dot{\gamma_{u}}}{L} \big| < \frac{16 \eta_{k}}{L}. 
	\end{align*}
\end{cor}
\begin{proof}
	Since for each $j$, we have
	\begin{align*}
	\sup_{u \in [\frac{j-1}{k}, \frac{j}{k}]} \bigg| \frac{\dot{\gamma_{u}}}{L} - \bigg( \frac{\Delta_{j} x}{|\Delta_{j} \gamma|}, \frac{\Delta_{j} y}{|\Delta_{j} \gamma|} \bigg) \bigg| < \frac{4 \eta_{k}}{L}, 
	\end{align*}
	it suffices to prove the inequality \eqref{eq:accuracy_signed_direction}. If $r_{j} \in [\frac{2 \eta_{k}}{L}, 1 - \frac{2 \eta_{k}}{L}]$, then by Definition \ref{de:finding_signs} and Theorem \ref{th:signs}, $a_{j}^{x}$ and $a_{j}^{y}$ have the same signs as $\Delta_{j} x$ and $\Delta_{j} y$, respectively. Thus, by Corollary \ref{cor:finding_directions}, we have
	\begin{equation} \label{eq:accuracy_x}
	\bigg|a_{j}^{x} \rho_{j} - \frac{\Delta_{j} x}{|\Delta_{j} \gamma|} \bigg| = |\rho_{j} - r_{j}| < \frac{4 \eta_{k}}{L}, 
	\end{equation}
	and the same bound is true for $\big|a_{j}^{y} (1-\rho_{j}) - \frac{\Delta_{j} y}{|\Delta_{j} \gamma|} \big|$. This proves \eqref{eq:accuracy_signed_direction} when $r_{j} \in [\frac{2 \eta_{k}}{L}, 1 - \frac{2 \eta_{k}}{L}]$. 
	
	If $r_{j} > 1 - \frac{2 \eta_{k}}{L}$, the bound \eqref{eq:accuracy_x} for the $x$-direction is still true, but the bound for the $y$-direction becomes
	\begin{align*}
	\bigg|a_{j}^{y} (1 - \rho_{j}) - \frac{\Delta_{j} y}{|\Delta_{j} \gamma|} \bigg| \leq (1 - \rho_{j}) + (1 - r_{j}) <  \frac{8 \eta_{k}}{L}, 
	\end{align*}
	since Corollary \ref{cor:finding_directions} forces $\rho_{j}$ to be bigger than $1 - \frac{6 \eta_{k}}{L}$. 
	Combining this bound with \eqref{eq:accuracy_x}, we get \eqref{eq:accuracy_signed_direction} when $r_{j} > 1 - \frac{2 \eta_{k}}{L}$. The case when $r_{j} < \frac{2 \eta_{k}}{L}$ follows in the same way. This completes the proof. 
\end{proof}

\subsection{Length} \label{sec:length}

We have now recovered from the signature the piecewise linear path $\zeta$ with unit length, which is shown in Corollary \ref{cor:accuracy_signed_directions} to be close to the ``normalized" $\gamma$. Thus, the only remaining quantity to be determined is $\tilde{L}$, which is expected to approximate the $\ell^{1}$ length of $\gamma$. We can achieve this by a simple scaling argument. 

Let $\zeta$ be the unit piecewise linear path as in \eqref{eq:unit_approximation}. Let $m$ be the smallest integer such that $C_{\gamma}(\tilde{w}) \neq 0$ for some $|\tilde{w}| = m$, where $C_{\gamma}(\tilde{w})$ is the coefficient of the word $\tilde{w}$ in the signature of $\gamma$. Fix that word $\tilde{w}$, and set
\begin{equation} \label{eq:length}
\tilde{L} := \bigg( \frac{C_{\gamma}(\tilde{w})}{C_{\zeta}(\tilde{w})} \bigg)^{\frac{1}{m}}. 
\end{equation}
Note that it is not obvious from the above expression that $\tilde{L}$ is always well defined, as the denominator might just be $0$, or $C_{\zeta}(\tilde{w})$ has a different sign with $C_{\gamma}(\tilde{w})$ and $m$ is even. However, it turns out that for all sufficiently large $k$, $\tilde{L}$ defined by \eqref{eq:length} does make sense, and is in fact close to the true length $L$. This is the content of the following theorem. 

\begin{thm} \label{th:length}
Let $k$ be large enough such that
\begin{equation} \label{eq:condition_length}
\eta_{k} < \frac{(m-1)!}{32 L^{m-1}} |C_{\gamma}(\tilde{w})|. 
\end{equation}
Then $\tilde{L}$ introduced in \eqref{eq:length} is well defined, and satisfies
\begin{equation} \label{eq:length_bound}
|\tilde{L} - L| < \frac{32 L^{m} \eta_{k}}{(m-1)! |C_{\gamma}(\tilde{w})|}. 
\end{equation}
\end{thm}
\begin{proof}
	If $\theta$ is at natural parametrization, then we have
	\begin{align*}
	\dot{\zeta}_{v} = \theta_{j}, \qquad v \in \big( \frac{j-1}{k}, \frac{j}{k} \big). 
	\end{align*}
	If $\gamma$ is also at natural parametrization, then for almost every $\uu = (u_{1}, \dots, u_{m}) \in \Delta_{m}$, we have
	\begin{align*}
	\bigg|\frac{1}{L^{m}} \prod_{j=1}^{m} \dot{\gamma}_{u_{j}}^{i_{j}} - \prod_{j=1}^{m} \dot{\zeta}_{u_{j}} \bigg| \leq \frac{16 m \eta_{k}}{L}, 
	\end{align*}
	where we have used Corollary \ref{cor:accuracy_signed_directions} and the fact that both $\frac{|\dot{\gamma}|}{L}$ and $|\dot{\zeta}|$ are uniformly bounded by $1$. Thus, the difference between the signatures $\frac{1}{L^{m}} C_{\gamma}(\tilde{w})$ and $C_{\zeta}(\tilde{w})$ can be bounded by
	\begin{equation} \label{eq:difference_signature}
	\big| \frac{1}{L^{m}} C_{\gamma}(\tilde{w}) - C_{\zeta}(\tilde{w}) \big| \leq \frac{16 \eta_{k}}{(m-1)! L}. 
	\end{equation}
	Note that \eqref{eq:condition_length} and \eqref{eq:difference_signature} together imply $C_{\zeta}(\tilde{w})$ has the same sign as $C_{\gamma}(\tilde{w})$, and $|C_{\zeta}(\tilde{w})| > \frac{|C_{\gamma}(\tilde{w})|}{2 L^{m}}$. In particular, this shows $\tilde{L}$ in \eqref{eq:length} is well defined and positive, so we have
	\begin{equation} \label{eq:length_high_power}
	|\tilde{L}^{m} - L^{m}| \geq L^{m-1} |\tilde{L} - L|.  
	\end{equation}
	On the other hand, the lower bound $|C_{\zeta}(\tilde{w})| > \frac{|C_{\gamma}(\tilde{w})|}{2 L^{m}}$ implies
	\begin{equation} \label{eq:signature_ratio}
	|\tilde{L}^{m} - L^{m}| = \bigg| \frac{C_{\gamma}(\tilde{w})}{C_{\zeta}(\tilde{w})} - L^{m} \bigg| \leq \frac{ 16 L^{m-1} \eta_{k}}{(m-1)! |C_{\theta}(\tilde{w})|} < \frac{ 32 L^{2m-1} \eta_{k}}{(m-1)! |C_{\gamma}(\tilde{w})|}. 
	\end{equation}
	Combining \eqref{eq:length_high_power} and \eqref{eq:signature_ratio}, we obtain
	\begin{align*}
	|\tilde{L} - L| < \frac{32 L^{m} \eta_{k}}{(m-1)! |C_{\gamma}(\tilde{w})|}. 
	\end{align*}
	This proves the theorem. 
\end{proof}

\begin{rmk}
	Note that both the assumption \eqref{eq:condition_length} and the bound \eqref{eq:length_bound} give the correct scaling in length, since $\eta_{k}$ scales linearly in $L$, and $|C_{\gamma}(\tilde{w})|$ scales as $L^{m}$. 
\end{rmk}

\subsection{Summary of the procedure}

We now summarize the inversion procedure developed in this section, and give its validity as well as stability properties. Let $k$ be a fixed large number whose value will be specified in Theorem \ref{th:main} below, and let $n = 4k^{2} \log k$. Recall the definition of $\eta_{k}$ from \eqref{eq:eta}. Also recall the definitions of the set $\lL^{n}_{k}$ from \eqref{eq:index_set} and the symmetrized signatures $\sS^{2n}_{k}$, $\sS^{2n}_{k,j,x}$, $\sS^{2n}_{k,j,y}$ from \eqref{eq:symmetrized_sig} and \eqref{eq:symmetrized_sig_sign}. The inversion procedure includes choosing for each $j = 1, \dots, k$ a real number $\rho_{j} \in [0,1]$, $a_{j}^{x}, a_{j}^{y} \in \{ -1, 1\}$ and $\tilde{L} \in \RR^{+}$ in the following way.

\begin{enumerate}

\item For each $j = 1, \dots, k$, choose $\rho_{j} \in [0,1]$ according to Corollary \ref{cor:finding_directions} such that
\begin{equation} \label{eq:summary_direction}
\bigg(\sum_{|w|=k-1} \sum_{\Bell: |\frac{\ell_{j}}{n} - \rho_{j}| \leq \frac{2 \eta_{k}}{L}} \big| \sS^{2n}_{k}(w, \Bell)| \bigg) \bigg/ \bigg(\sum_{|w|=k-1} \sum_{\Bell \in \lL^{n}_{k}} |\sS^{2n}_{k}(w, \Bell)| \bigg)  > \frac{1}{2}, 
\end{equation}
where in the numerator the sum over $\Bell$ is restricted to the indicated subset of $\lL^{n}_{k}$. Any set $\{\rho_{j}\}$ satisfying \eqref{eq:summary_direction} can be used\footnote{Note that this step uses the knowledge of $\frac{\eta_{k}}{L}$, which is not directly available from the signature. However, we can still circumvent the problem even if we do not know $\frac{\eta_{k}}{L}$, at the cost of a lower accuracy of the inversion procedure. See Remark \ref{rm:prior_information} for a discussion on this. }. 

\medskip

\item Now we choose a word $w^{*} = e_{i_{1}} \cdots e_{i_{k-1}}$ by setting
\begin{align*}
e_{i_{j}} = x, \qquad \text{if} \phantom{1} \rho_{j} \geq \frac{1}{2}; \\
e_{i_{j}} = y, \qquad \text{if} \phantom{1} \rho_{j} < \frac{1}{2}, 
\end{align*}
and determine the signs $a_{j}^{x}$'s by
\begin{align*}
a_{j}^{x} = 1, \qquad \text{if} \phantom{1} \frac{\sum_{\Bell \in \lL^{n}_{k}} \sS_{k}^{2n}(w^{*}, \Bell)}{\sum_{\Bell \in \lL^{n}_{k}} \sS_{k,j,x}^{2n}(w^{*}, \Bell)} \geq 0; \\
a_{j}^{x} = -1, \qquad \text{if} \phantom{1} \frac{\sum_{\Bell \in \lL^{n}_{k}} \sS_{k}^{2n}(w^{*}, \Bell)}{\sum_{\Bell \in \lL^{n}_{k}} \sS_{k,j,x}^{2n}(w^{*}, \Bell)} < 0,  
\end{align*}
where $w^{*}$ is the word chosen above. The signs $a_{j}^{y}$'s are determined in the same way except one replaces $\sS^{2n}_{k,j,x}(w^{*}, \Bell)$ by $\sS^{2n}_{k,j,y}(w^{*}, \Bell)$. These two steps already produce a piecewise linear path
\begin{equation} \label{eq:summary_unit_path}
\zeta = \frac{1}{k} \big( \theta_{1} * \cdots * \theta_{k} \big)
\end{equation}
as in \eqref{eq:unit_approximation}, where $\theta_{j} = \big( a_{j}^{x} \rho_{j}, a_{j}^{y}(1-\rho_{j}) \big)$ for each $j = 1, \dots, k$. 

\medskip

\item Let $m$ be the smallest integer such that $C_{\gamma}(\tilde{w}) \neq 0$ for some $|\tilde{w}| = m$, and we determine the length $\tilde{L}$ by setting
\begin{equation} \label{eq:summary_length}
\tilde{L} := \bigg( \frac{C_{\gamma}(\tilde{w})}{C_{\zeta}(\tilde{w})} \bigg)^{\frac{1}{m}}, 
\end{equation}
where $\theta$ is the path chosen in \eqref{eq:summary_unit_path} above, and $C_{\gamma}(\tilde{w})$ and $C_{\zeta}(\tilde{w})$ are the coefficients of $\tilde{w}$ in the signature of $\gamma$ and $\theta$, respectively. 
\end{enumerate}

The above three steps produce a piecewise linear path $\tilde{\gamma}$ of the form
\begin{align*}
\tilde{\gamma} = \tilde{L} \zeta = \frac{\tilde{L}}{k} \big( \theta_{1} * \cdots \theta_{k }\big). 
\end{align*}
Note that given any set $\rho_{j} \in [0,1]$, the choices for $a_{j}^{x}$ and $a_{j}^{y}$ are always well defined. For large enough $k$, the choices for the parameters $\{\rho_{j}\}$ and $\tilde{L}$ are also well defined, and the path $\tilde{\gamma}$ will turn out to be close to the original path $\gamma$ in Lipschitz norm. 

We first quantify how large $k$ needs to be. All the statements before Section \ref{sec:length} (recovery of length) hold true when $\eta_{k} < \frac{L}{32}$. In order for $\tilde{L}$ to be well defined (Theorem \ref{th:length}), one needs $\eta_{k}$ to satisfy the condition \eqref{eq:condition_length}. But since 
\begin{align*}
|C_{\gamma}(w)| \leq \frac{L^{m}}{m!}
\end{align*}
for every $m$ and every $|w| = m$, so the assumption \eqref{eq:condition_length} implies $\eta_{k} < \frac{L}{32}$. We are now ready to state our main theorem. 

\begin{thm} \label{th:main}
	Let $m$ be the smallest integer such that $C_{\gamma}(\tilde{w}) \neq 0$ for some $|\tilde{w}| = m$, and we fix that word $\tilde{w}$. Let $k \geq 4$ be large enough such that
	\begin{align*}
	\eta_{k} < \frac{(m-1)!}{32 L^{m-1}} |C_{\gamma}(\tilde{w})|, 
	\end{align*}
	and let $n = 4 k^{2} \log k$. Then, the above choices of parameters are well defined. In addition, when $\gamma$ is at natural parametrization, we have the bound
	\begin{align*}
	\sup_{1 \leq j \leq k} \sup_{u \in [\frac{j-1}{k}, \frac{j}{k}]} |\dot{\gamma}_{u} - \tilde{L} \theta_{j}| < 16 \eta_{k} \big( 1 + \frac{2 L^{m}}{(m-1)! |C_{\gamma}(\tilde{w})|} \big), 
	\end{align*}
	where $\theta_{j} = \big( a_{j}^{x} \rho_{j}, a_{j}^{y}(1-\rho_{j}) \big)$. 
\end{thm}
\begin{proof}
	By Corollary \ref{cor:accuracy_signed_directions}, for every $j = 1, \dots, k$, we have
	\begin{align*}
	\sup_{u \in [\frac{j-1}{k}, \frac{j}{k}]} \big| \dot{\gamma}_{u}- L \theta_{j} \big| < 16 \eta_{k}. 
	\end{align*}
	By the bound on $|\tilde{L} - L|$ from Theorem \ref{th:length} and that $|\theta_{j}| = 1$, we get
	\begin{align*}
	|\dot{\gamma}_{u} - \tilde{L}\theta_{j}| < 16 \eta_{k} \big( 1 + \frac{2 L^{m}}{(m-1)! |C_{\gamma}(\tilde{w})|} \big), 
	\end{align*}
	which holds for all $u \in [\frac{j-1}{k}, \frac{j}{k}]$ and all $j = 1, \dots, k$. This completes the proof of the main theorem. 
\end{proof}

\section{Higher dimensions} \label{sec:higher_dimensions}

We now give a brief explanation of how the symmetrization procedure extends to the recovery of $\cC^{1}$ paths in dimension higher than $2$ from their signatures. Let $\gamma = (\gamma^{1}, \dots, \gamma^{d})$ be a $d$-dimensional $\cC^{1}$ path at natural parametrization. For large enough integer $k$ and each $1 \leq j \leq k$, we need to reconstruct from the signature of $\gamma$ a non-negative vector $\Brho_{j} = (\rho_{j}^{1}, \dots, \rho_{j}^{d})$ such that $\sum_{i} \rho_{j}^{i} = 1$, the signs $a_{j}^{i} \in \{\pm 1\}$ for $1 \leq i \leq d$, and $\tilde{L} \in \RR^{+}$ that approximates the $\ell^{1}$ length $L$ of $\gamma$. 

Same as the $2$-dimensional case, we still symmetrize $k$ blocks of size $2n$ with one letter separating consecutive blocks. To set up, for every positive integer $k$ and $n$, we let $\lL^{n}_{k}$ be the set of multi-component multi-indices
\begin{align*}
\lL^{n}_{k} := \big\{ \Bell = (\Bell_{1}, \dots, \Bell_{k}): \phantom{1} \Bell_{j} = (\ell_{j}^{1}, \dots, \ell_{j}^{d}), \sum_{i=1}^{d} \ell_{j}^{i} = n, \forall j = 1, \dots, k \big\}. 
\end{align*}
For every word $w = e_{i_{1}} \cdots e_{i_{k-1}}$ and every $\Bell \in \lL^{n}_{k}$, we define $\wW^{2n}_{k}(w, \Bell)$ to be the set of words of the form 
\begin{align*}
\wW^{2n}_{k}(w, \Bell) := \big\{ w' = w_{1} * e_{i_{1}} * \cdots * e_{i_{k-1}} * w_{k}:  |w_{j}|_{e_{i}} = 2 \ell_{j}^{i}, \forall i,j  \big\}, 
\end{align*}
where $e_{1}, \dots, e_{d}$ are standard basis elements of $\RR^{d}$. Note that the assumption $\Bell \in \lL^{n}_{k}$ together with $|w_{j}|_{e_{i}} = 2 \ell_{j}^{i}$ imply $|w_{j}| = 2n$ for every $j$. Also, for every integer $1 \leq p \leq k$ and $1 \leq q \leq d$, we let $\wW^{2n}_{k,p,q}(w, \Bell)$ be the set of words
\begin{align*}
w' = w_{1} * e_{i_{1}} * \cdots * e_{i_{k-1}} * w_{k}
\end{align*}
such that $|w_{j}|_{e_{i}} = 2 \ell_{j}^{i}$ if $j \neq p$ or $i \neq q$, but $|w_{p}|_{e_{q}} = 2 \ell_{p}^{q} + 1$. In other words, the definition for $\wW^{2n}_{k,p,q}(w, \Bell)$ is the same as $\wW^{2n}_{k}(w, \Bell)$ except that $|w_{p}|_{e_{q}} = 2 \ell_{p}^{q} + 1$, and as a consequence we have $|w_{p}| = 2n+1$. 

Similar to \eqref{eq:symmetrized_sig} and \eqref{eq:symmetrized_sig_sign}, we define the symmetrized signatures $\sS^{2n}_{k}(w, \Bell)$ and $\sS^{2n}_{k,p,q}(w,\Bell)$ by
\begin{equation} \label{eq:symmetrized_sig_multi}
\begin{split}
\sS^{2n}_{k}(w, \Bell) &= \big((2n)!\big)^{k} \sum_{w' \in \wW^{2n}_{k}(w,\Bell)} C(w'), \\
\sS^{2n}_{k,p,q}(w, \Bell) &= (2n+1) \big( (2n)! \big)^{k} \sum_{w' \in \wW^{2n}_{k,p,q}(w,\Bell)} C(w'). 
\end{split}
\end{equation}
Finally, for every integer $k$, we define $\epsilon_{k}$ and $\eta_{k}$ by
\begin{equation} \label{eq:eta_multi}
\epsilon_{k} := \sqrt{2} \bigg( \sqrt{\frac{(d-1) \delta(\frac{1}{k})}{L}} + \frac{1}{\sqrt{k}} \bigg), \qquad \eta_{k} := \delta (3 \epsilon_{k}) + \frac{L}{\sqrt{k}}, 
\end{equation}
where $L$ is the $\ell^{1}$ length of $\gamma$, and $\delta$ is the modulus of continuity of $\dot{\gamma}$. Now, we let $k$ to be a fixed integer whose value will be specified later, and let $n = 2 d k^{2} \log k$. Similar to the $2$-dimensional case, we have the following inversion algorithm to reconstruct $\gamma$ from its signature. 

\begin{enumerate}
	\item For each $1 \leq j \leq k$, we choose the unsigned direction
	\begin{align*}
	\Brho_{j} = (\rho_{j}^{1}, \dots, \rho_{j}^{d}) \phantom{11} \text{s.t.} \phantom{11} \rho_{j}^{i} \geq 0, \quad \sum_{i=1}^{d} \rho_{j}^{i} = 1
	\end{align*}
	according to the criterion
	\begin{align*}
	\bigg( \sum_{w:|w|=k-1} \sum_{\Bell: \sup_{i} |\frac{\ell_{j}^{i}}{n} - \rho_{j}^{i}| < \frac{3d \eta_{k}}{L}} \big| \sS^{2n}_{k}(w, \Bell) \big| \bigg) \bigg/ \bigg( \sum_{w:|w|=k-1}  \sum_{\Bell \in \lL^{n}_{k}} \big| \sS^{2n}_{k}(w, \Bell) \big| \bigg) > \frac{1}{2}. 
	\end{align*}
	Any element $\Brho_{j}$ satisfying the above will suffice. 
	
	\item We choose the word $w^{*} = e_{i_{1}} \cdots e_{i_{k-1}}$ by setting
	\begin{align*}
	e_{i_{j}} = e_{i}, \qquad \text{if} \phantom{1} \rho_{j}^{i} \geq \rho_{j}^{q}, \quad \forall q \neq i. 
	\end{align*}
	If there are two or more maximizers, then any of them is suitable. We then determine the signs $a_{j}^{i} \in \{ \pm 1 \}$ by
	\begin{align*}
	a_{j}^{i} = 1, \qquad &\text{if} \quad \frac{\sum_{\Bell} \sS^{2n}_{k}(w^{*}, \Bell)}{\sum_{\Bell} \sS^{2n}_{k,j,i}(w^{*}, \Bell)} \geq 0, \\
	a_{j}^{i} = -1, \qquad &\text{if} \quad \frac{\sum_{\Bell} \sS^{2n}_{k}(w^{*}, \Bell)}{\sum_{\Bell} \sS^{2n}_{k,j,i}(w^{*}, \Bell)} < 0, 
	\end{align*}
	where all the sums are taken over $\Bell \in \lL^{n}_{k}$. 
	
	\item The recovery of length is exactly the same as the $2$-dimensional case. Let
	\begin{align*}
	\theta_{j} = \big( a_{j}^{1} \rho_{j}^{1}, \dots, a_{j}^{d} \rho_{j}^{d} \big)
	\end{align*}
	be the unit vector (in $\ell^{1}$ sense) obtained from the previous two steps, and let
	\begin{align*}
	\zeta := \frac{1}{k} (\theta_{1} * \cdots * \theta_{k}). 
	\end{align*}
	In other words, $\zeta$ is a piecewise linear path whose $j$-th piece is in the direction $\theta_{j}$ and has length $\frac{1}{k}$. We also let $m$ be the smallest integer such that $C_{\gamma}(\tilde{w}) \neq 0$ for some $|\tilde{w}| = m$, and determine $\tilde{L}$ by
	\begin{align*}
	\tilde{L} := \bigg( \frac{C_{\gamma}(\tilde{w})}{C_{\zeta}(\tilde{w})} \bigg)^{\frac{1}{m}}, 
	\end{align*}
	where $\tilde{w}$ can be any word with length $m$ such that $C_{\gamma}(\tilde{w}) \neq 0$. 
\end{enumerate}

Following the arguments for the $2$-dimensional case, we can get the following theorem, which gives the stability of the inversion algorithm for $\cC^{1}$ paths in $\RR^{d}$. 

\begin{thm} \label{th:multi}
	Let $m$ be the smallest integer such that $C_{\gamma}(\tilde{w}) \neq 0$ for some $|\tilde{w}| = m$. Let $k \geq 2d$ be large enough such that
	\begin{align*}
	\eta_{k} < \frac{(m-1)! |C_{\gamma}(\tilde{w})|}{16(d+1) L^{m-1}}, 
	\end{align*}
	and let $n = 2d k^{2} \log k$. Then for every $1 \leq j \leq k$, the above choices of the parameters $\Brho_{j}$, $\{a_{j}^{i}\}_{i=1}^{d}$ and $\tilde{L}$ are all well defined, and when $\gamma$ is at natural parametrization, we have the bound
	\begin{align*}
	\sup_{1 \leq i \leq d} \sup_{1 \leq j \leq k} \sup_{u \in [\frac{j-1}{k}, \frac{j}{k}]} |\dot{\gamma}_{u}^{i} - \tilde{L} a_{j}^{i} \rho_{j}^{i}| < 8(d+1) \big( 1 + \frac{2 L^{m}}{(m-1)! |C_{\gamma}(\tilde{w})|} \big) \eta_{k}. 
	\end{align*}
	In other words, the reconstructed path
	\begin{align*}
	\tilde{L} \zeta = \frac{\tilde{L}}{k} \big( \theta_{1} * \cdots * \theta_{k} \big)
	\end{align*}
	is $\eta_{k}$-close to the original path $\gamma$ in Lipschitz norm. 
\end{thm}

\begin{rmk}
	The proportionality constants in Theorem \ref{th:main} for the $2$-dimensional case are better than the ones in Theorem \ref{th:multi} by simply setting $d=2$. This is because when $d=2$, the assumption $|\dot{x}| + |\dot{y}| \equiv L$ implies that in many cases (for example, Lemma \ref{le:directions}), the same optimal estimates hold for both $x$ and $y$ directions, while this is not true in general case when $d \geq 3$. 
\end{rmk}

\bibliographystyle{Martin}
\bibliography{Refs}

\bigskip

\textsc{Mathematical and Oxford-Man Institutes, University of Oxford, Woodstock road, Oxford, OX2 6GG, UK}. 

Email: tlyons@maths.ox.ac.uk

\smallskip

\textsc{Mathematics Institute, University of Warwick, Coventry, CV4 7AL, UK}. 

Email: weijun.xu@warwick.ac.uk

\end{document}